\pdfoutput=1
\documentclass[a4paper,10pt]{amsart}
\usepackage{
 amsmath,
 amssymb,
 mathrsfs,
 mathtools,
 fullpage,
 thmtools,
 stmaryrd,
 color,
 tikz-cd
}
\usepackage[shortlabels]{enumitem}

\usepackage[l2tabu, orthodox]{nag}
\usepackage[pagebackref]{hyperref}

\usepackage[capitalize]{cleveref}
\usetikzlibrary{arrows}

\DeclareFontFamily{U}{wncy}{}
\DeclareFontShape{U}{wncy}{m}{n}{<->wncyr10}{}
\DeclareSymbolFont{mcy}{U}{wncy}{m}{n}
\DeclareMathSymbol{\Sha}{\mathord}{mcy}{"58}

\theoremstyle{plain}
 \newtheorem{theorem}{Theorem}[subsection]
 \newtheorem{lemma}[theorem]{Lemma}
 \newtheorem{proposition}[theorem]{Proposition}
  \newtheorem{conjecture}[theorem]{Conjecture}
 \newtheorem{assumption}[theorem]{Assumption}
 \newtheorem{assumptions}[theorem]{Assumptions}
  \newtheorem{corollary}[theorem]{Corollary}
 \newtheorem{definition}[theorem]{Definition}
 \newtheorem{notation}[theorem]{Notation}
 \newtheorem{ltheorem}{Theorem} 
 
 \newtheorem{hypothesis}[theorem]{Hypothesis}

\declaretheorem[name=Proposition,sibling=theorem,qed={\qedsymbol}]{propqed}

\theoremstyle{remark}
 \declaretheorem[name=Remark,sibling=theorem,qed={\lower-0.3ex\hbox{$\diamond$}}]{remark}
 \declaretheorem[name=Note,sibling=theorem,qed={\lower-0.3ex\hbox{$\diamond$}}]{note}
 \declaretheorem[name=Remark,qed={\lower-0.3ex\hbox{$\diamond$}}]{remarknono}

\newcommand{\Ql}{\QQ_\ell}

\newcommand{\ep}{\epsilon}

\newenvironment{smatrix}{\left( \begin{smallmatrix} } {\end{smallmatrix} \right) }
\newcommand{\stbt}[4]{\begin{smatrix}#1 & #2 \\ #3 & #4\end{smatrix}}

\newcommand{\upi}{\underline{\pi}}

\newcommand{\uet}{\underline{\eta}}

\newcommand{\Dcris}{\mathbf{D}_{\mathrm{cris}}}

\DeclareMathOperator{\GSp}{GSp}
\DeclareMathOperator{\Kl}{Kl}
\DeclareMathOperator{\Iw}{Iw}
\DeclareMathOperator{\fs}{fs}

\DeclareMathOperator{\Spec}{Spec}
\DeclareMathOperator{\Fil}{Fil}
\DeclareMathOperator{\GL}{GL}
\DeclareMathOperator{\Gr}{Gr}

\DeclareMathOperator{\Gal}{Gal}
\DeclareMathOperator{\Res}{Res}
\DeclareMathOperator{\Ind}{Ind}\DeclareMathOperator{\ES}{ES}
\DeclareMathOperator{\cusp}{cusp}
\DeclareMathOperator{\Sieg}{Si}

\DeclareMathOperator{\an}{an}

\DeclareMathOperator{\End}{End}

\newcommand{\new}{\mathrm{new}}

\newcommand{\cE}{\mathcal{E}}
\newcommand{\cF}{\mathcal{F}}
\newcommand{\cG}{\mathcal{G}}
\newcommand{\cH}{\mathcal{H}}

\newcommand{\cL}{\mathcal{L}}
\newcommand{\cM}{\mathcal{M}}
\newcommand{\cN}{\mathcal{N}}
\newcommand{\cO}{\mathcal{O}}

\newcommand{\cQ}{\mathcal{Q}}
\newcommand{\cR}{\mathcal{R}}
\newcommand{\cS}{\mathcal{S}}
\newcommand{\cT}{\mathcal{T}}

\newcommand{\cW}{\mathcal{W}}
\newcommand{\cX}{\mathcal{X}}

\newcommand{\fS}{\mathfrak{S}}

\newcommand{\br}{\mathbf{r}}
\newcommand{\bt}{\mathbf{t}}
\newcommand{\bz}{\mathbf{z}}\newcommand{\bj}{\mathbf{j}}

\newcommand{\CC}{\mathbf{C}}
\newcommand{\QQ}{\mathbf{Q}}

\newcommand{\ZZ}{\mathbf{Z}}
\renewcommand{\AA}{\mathbf{A}}

\newcommand{\Qp}{\QQ_p}
\newcommand{\Zp}{\ZZ_p}
\newcommand{\Af}{\AA_{\mathrm{f}}}
\newcommand{\pif}{\pi_{\mathrm{f}}}

\newcommand{\id}{\mathrm{id}}

\newcommand{\into}{\hookrightarrow}

\renewcommand{\le}{\leqslant}

\renewcommand{\ge}{\geqslant}
\renewcommand{\geq}{\geqslant}

\DeclareMathOperator{\As}{As}

\DeclareMathOperator{\ord}{ord}

\DeclareMathOperator{\SL}{SL}

\newcommand{\GQS}{\Gamma_{\QQ, S}}
\newcommand{\rb}{\bar{\rho}}

\newcommand{\QQbar}{\overline{\QQ}}

\newcommand{\uchi}{\underline{\chi}}

\newcommand{\ad}{\mathrm{ad}}

\newcommand{\B}{\mathrm{B}}
\newcommand{\wB}{\mathrm{wB}}
 \newcommand{\wKl}{\mathrm{wKl}}

\numberwithin{equation}{section}

\author{David Loeffler}
\address{David Loeffler, Mathematics Institute\\
 University of Warwick\\
 Coventry CV4 7AL, UK.}
\email{d.a.loeffler@warwick.ac.uk}
\urladdr{\href{http://orcid.org/0000-0001-9069-1877}{0000-0001-9069-1877}}

\author{Sarah Livia Zerbes}
\address{Sarah Livia Zerbes, Department of Mathematics\\
 University College London\\
London WC1E 6BT, UK.}
\email{sarah.zerbes@math.ethz.ch}
\curraddr{Department of Mathematics, ETH Z\"urich, R\"amistrasse 101, 8092 Zurich, Switzerland}
\urladdr{\href{http://orcid.org/0000-0001-8650-9622}{0000-0001-8650-9622}}

\thanks{Supported by the following grants: ERC Consolidator Grant  \#101001051 ``Shimura varieties and the BSD conjecture'' (Loeffler); EPSRC New Horizons Grant (Zerbes).}

\title{On the Birch--Swinnerton-Dyer conjecture for modular abelian surfaces}

\begin{document}

 \begin{abstract}
  Let $A$ be a modular abelian surface over $\QQ$ which either has trivial geometric endomorphism ring, or arises as the restriction of scalars of an elliptic curve over an imaginary quadratic field which is modular and is not a $\QQ$-curve. In the former case, assume that there exists an odd Dirichlet character $\chi$ such that $L(A,\chi,1)\neq 0$. We prove the following implication: if $L(A, 1) \ne 0$, and the $p$-adic eigenvariety for $\GSp_4$ is smooth at the point corresponding to $A$ (and some auxiliary technical hypotheses hold), then $A(\QQ)$ is finite, as predicted by the Birch--Swinnerton-Dyer conjecture, and the $p$-part of the Tate--Shafarevich group is also finite. We also prove one inclusion of the cyclotomic Iwasawa Main Conjecture for $A$. Moreover, we also prove analogous results for cohomological automorphic representations of $\GSp_4$, removing many of the restrictive hypotheses in our earlier work \cite{LZ20}; for cohomological representations we do not need to assume smoothness of the eigenvariety, since it is automatic in this case. The main ingredient in the proof is the Euler system attached to the spin representations of genus $2$ Siegel modular forms constructed in our earlier work with Skinner \cite{LSZ17}.
 \end{abstract}

\maketitle


\section{Introduction}

 \subsection{Abelian surfaces} One of the most famous open problems in number theory is the Birch--Swinnerton-Dyer conjecture, which relates the arithmetic of a rational elliptic curve $E$ to the properties of the Hasse-Weil $L$-function $L(E,s)$: it predicts that  $\operatorname{rank}_{\mathbf{Z}}E(\QQ) = \ord_{s=1}L(E,s)$. The conjecture has a straightforward generalisation to elliptic curves defined over a number field $K$ (and more generally abelian varieties defined over $K$), assuming analytic continuation of the $L$-function to the point $s=1$.

 One of the strongest results on the BSD conjecture for elliptic curves over $\QQ$ is due to Kolyvagin \cite{kolyvagin90}: using the first example of what is now called an Euler system, he proved that if $L(E,1)\neq 0$, then $\operatorname{rank}_{\ZZ} E(\QQ)=0$. In this paper, we prove a similar result for abelian surfaces $A$ over $\QQ$, although our result is conditional on several other conjectures (which we hope should be easier to attack than BSD itself).

 \begin{ltheorem}\label{thm:maintheorem}
  Let $A/\QQ$ be an abelian surface, and suppose that $A$ satisfies the following hypotheses:
  \begin{enumerate}
   \item $L(A, 1) \ne 0$.
   \item\label{endZ} Either $A$ is generic (that is, $\End_{\QQbar}(A) = \ZZ$); or $A = \Res_{K / \QQ}(E)$ where $K$ is an imaginary quadratic field, and $E$ is an elliptic curve over $K$ which is not a $\QQ$-curve (i.e.~is not CM and is not $\QQbar$-isogenous to its Galois conjugate).
   \item\label{modular} $A$ is modular, associated to a cuspidal automorphic representation of $\GSp_4 / \QQ$.
   \item\label{rohrlich} there exists an odd Dirichlet character $\chi_-$ such that $L(A, \chi_-, 1) \ne 0$.
  \end{enumerate}

  Let $p$ be a prime such that the following conditions hold:
  \begin{enumerate}[resume]
   \item\label{ord} $A$ has good ordinary reduction at $p$;

   \item\label{BI} For any Dirichlet character $\chi$, the 4-dimensional $p$-adic Galois representation $V_{p}(A)(\chi)$ satisfies the ``big image'' conditions of \cite[\S 3.5]{mazurrubin04};

   \item\label{def} the automorphic representation $\pi$ associated to $A$ is \emph{deformable} at $p$, in the sense of \cref{def:deformable}.
  \end{enumerate}

  Then $A(\QQ)$ and $\Sha_{p^\infty}(A / \QQ)$ are finite. If $A = \Res_{K / \QQ}(E)$, then $E(K)$ and $\Sha_{p^\infty}(E / K)$ are finite.
 \end{ltheorem}

 We also obtain one inclusion of the cyclotomic Iwasawa Main Conjecture for $A$, assuming hypotheses (2)--(7); see \cref{thm:IMC}.

 \begin{remark}
  The Birch--Swinnerton-Dyer conjecture predicts, of course, that hypothesis (1) alone should be sufficient for the conclusion of Theorem A. Let us discuss the remaining hypotheses:

  \begin{itemize}

  \item Hypothesis (\ref*{endZ}) is not a major restriction, since the excluded cases correspond to automorphic forms for smaller groups, and thus should be treatable by existing methods (e.g.~see \cite{kolyvaginlogachev89} for the case of $\GL_2$-type abelian surfaces). For the case of inductions from \emph{real} quadratic fields, see \cite{LZ20-yoshida}.

  \item The Paramodular Conjecture \cite{brumerkramer14} predicts that (\ref*{endZ}) should imply (\ref*{modular}), and many partial results are known; for instance, it is shown in \cite{BCGP} that (\ref*{modular}) holds for infinitely many $\QQbar$-isomorphism classes of generic abelian surfaces over $\QQ$. For explicit examples, see \cite{CCG20} and the references therein. For lifts from imaginary quadratic fields, much stronger results are available, since the modularity of $E$ is known in many cases, including a positive proportion of elliptic curves $E$ over any given $K$ \cite{allenkharethorne19}; and this implies modularity of the Weil restriction using functorial lifts from $\GL_2 / K$ to $\GSp_4 / \QQ$.

  \item Hypothesis (\ref*{rohrlich}) is a special case of more general non-vanishing conjectures for twists of automorphic $L$-functions (see the introduction of \cite{DJR20} for a survey). These conjectures are wide open in general; however, in the case of inductions from $\GL_2 / K$ with $K$ quadratic imaginary, hypothesis (\ref*{rohrlich}) follows from hypothesis (1), since we can take $\chi_-$ to be the quadratic character associated to $K$.

  \item Under our assumptions on $A$, hypothesis (\ref*{ord}) is known to hold for a density 1 set of primes $p$, by \cite{sawin16}. Similarly, condition (\ref*{BI}) holds for almost all $p$: in the generic case the image of Galois is $\GSp_4(\Zp)$ for all but finitely many $p$ (by results of of Serre quoted in \cite[\S 4]{calegarigeraghty20}), from which the big-image condition is immediate, and the case of inductions from imaginary quadratic fields follows from Serre's results for $E / K$, exactly as in \cite[\S 2.5]{LZ20-yoshida} in the real-quadratic case.

  \item This leaves hypothesis (\ref*{def}), which is a little more awkward. We conjecture that this should hold for all $p$, and we explain in an appendix to this paper how this is related to existing conjectures on the geometry of eigenvarieties. In particular, we sketch an argument showing that if $A$ is induced from an imaginary quadratic field $K$, then a conjecture of Calegari--Mazur \cite{calegarimazur09} implies that $A$ should be deformable at all primes split in $K$.\qedhere
  \end{itemize}

 \end{remark}

 \subsection{The Bloch--Kato conjecture}

  \cref{thm:maintheorem} will be deduced as a special case of the following, more general theorem, regarding the Bloch--Kato conjecture for automorphic representations of $\GSp_4$.

  \begin{ltheorem}\label{thm:mainBK}
   Let $\pi$ be a globally generic, cuspidal automorphic representation of $\GSp_4$, with $\pi_\infty$ a discrete-series or limit of discrete-series representation of infinitesimal character $(r_1 + 2, r_2 + 1; -r_1-r_2)$ for some integers $r_1 \ge r_2 \ge -1$. Let $\Pi$ be the unitary twist of $\pi$, and suppose the following conditions hold, for some prime $p \ge 3$:
   \begin{enumerate}
   \item $\pi_p$ is unramified and Borel-ordinary, with respect to some choice of embedding $\iota: \QQ(\pi) \into \overline{\QQ}_p$.
   \item $\pi$ is deformable in the sense of \cref{def:deformable} (which is automatic if $r_2 \ge 0$).
   \item The 4-dimensional $p$-adic Galois representation $V_p(\pi)^*(\chi)$ satisfies the ``big image'' condition for any $\chi$;
   \item If $r_1 = r_2$, then there exists an odd Dirichlet character $\chi$ such that $L(\Pi \times \chi, \tfrac{1}{2}) \ne 0$.
   \end{enumerate}
   Let $j$ be an integer with $0 \le j \le r_1 - r_2$. If $L(\Pi, \tfrac{1-r_1+r_2}{2} + j) \ne 0$, then the Bloch--Kato Selmer group $H^1_{\mathrm{f}}(\QQ, V_p(\pi)^*(-1-r_2-j))$ vanishes.
  \end{ltheorem}

  This is a considerable strengthening of the main theorem of \cite{LZ20}, in which $\pi$ was assumed to be unramified at all finite places, and to have highly regular weight (in which case conditions (2) and (4) are automatically satisfied). Our new argument dispenses with these auxiliary hypotheses to give a considerably more general result. As in the abelian-surface case, our methods also give one inclusion in the Iwasawa main conjecture for $V_p(\pi)^*$ (Theorem \ref{thm:IMCautorep}).

 \subsection{Outline of the argument}

  The main tool for proving Theorems \ref{thm:maintheorem} and \ref{thm:mainBK} is an Euler system attached to the Galois representations $V_p(\pi)^*$ for cuspidal automorphic representations of $\GSp_4$. In \cite{LSZ17}, such an Euler system is constructed for the representations $V_{p}(\pi)^*$ associated to representations $\pi$ of \emph{regular weight}, i.e.~with $r_1 \ge r_2\geq 0$. To incorporate the non-regular weight case $r_2 = -1$, we use our subsequent work with Rockwood \cite{LRZ}, showing that the construction of \cite{LSZ17} extends to give $p$-adic families of Euler systems associated to Hida families of automorphic representations. If we can put our representation $\pi$ into such a family, then we can obtain an Euler system for $V_p(\pi)^*$ by specialising this family of Euler systems at the point corresponding to $\pi$.

  In order to use the Euler system to bound Selmer groups, we need to prove an explicit reciprocity law, relating the bottom class of the Euler system to the values of $L(\Pi, s)$. In turn, this explicit reciprocity law requires the existence of a three-variable $p$-adic $L$-function for the Hida family through $\pi$ (with $r_1, r_2$, and the cyclotomic twist all varying independently). In \cite{LZ20}, this was deduced from the results of \cite{barreradimitrovwilliams} (and its forthcoming sequel) applied to the functorial lift of $\pi$ to $\GL_4$. However, this cannot be applied in our setting, since when $r_2 = -1$ (as in the abelian-surface case), the lift of $\pi$ does not contribute to Betti cohomology, and hence the critical values of its $L$-function are invisible to the constructions of \cite{barreradimitrovwilliams}.

  In this paper, we rely instead on a different and much more flexible construction of $p$-adic $L$-functions in families, relying on the functoriality results for higher Coleman theory established in \cite{LZ21-erl}. This allows us to permit arbitrary ramification at the primes away from $p$, and -- much more importantly -- to access the critical $L$-values of non-regular-weight automorphic representations. This requires a very delicate study of the spectral sequences of higher Coleman theory, in order to understand the local behaviour of families around a singular-weight classical point; this is where the ``deformability'' condition arises.

  (Note that the proof of the explicit reciprocity law relies on the vanishing statements for Hecke eigenspaces in the rigid cohomology of certain strata in Siegel modular threefolds of Klingen level at $p$, formulated in \cite{LZ20}, whose proofs will appear in forthcoming work of Lan--Skinner. It also relies on the endoscopic classification of $\GSp_4$ automorphic representations due to Arthur and Gee--Ta\"ibi, which uses certain cases of the fundamental lemma whose proofs have not yet appeared; see \cite[\S 1.4]{BCGP} for further details.)

  One we have the explicit reciprocity law in hand, we apply the Euler system machine and a leading term argument to obtain a bound on the $p$-Selmer group of $\pi$. This gives Theorem \ref{thm:mainBK}; and Theorem \ref{thm:BSD} follows directly from this using the standard exact sequence relating the Selmer group to the Mordell--Weil and Tate--Shafarevich groups.

  \begin{remarknono}
   The proofs of \cref{thm:maintheorem,thm:IMC} can easily be generalised to the case of modular abelian surfaces twisted by $2$-dimensional odd Artin representations. In this case, the underlying Euler system is the one constructed in \cite{HJS20}, which is related to critical values of the degree $8$ $L$-function of  automorphic representations of $\GSp_4\times\GL_2$ (see \cite{LZ21-erl}). In order to prove cases of the equivariant BSD conjecture in analytic rank $0$, one needs to show that both the Euler system and the explicit reciprocity law deform to non-regular weights. The deformation of the Euler system follows from \cite{LRZ}, and the deformation of the explicit reciprocity law can be proved using the methods introduced in this paper.
  \end{remarknono}

  \begin{remarknono}
   This work was inspired by the beautiful article \cite{BDR-BeilinsonFlach2} of Bertolini, Darmon and Rotger, where they use $p$-adic deformation of the Beilinson--Flach Euler system to prove many cases of the equivariant BSD conjecture for elliptic curves over $\QQ$.
  \end{remarknono}

  \emph{Acknowledgements.} We are very grateful to George Boxer and Vincent Pilloni for answering our questions on higher Coleman theory, and to Frank Calegari for a helpful discussion about the local geometry of eigenvarieties. We would also like to thank John Coates for his interest and encouragement, ever since we started working on this project in 2015. Finally, we thank Toby Gee and Mark Kisin for their remarks on an earlier draft of this paper.


\section{Setup and notation}

 \subsection{Groups}

  Let $G=\GSp_4$, with respect to the anti-diagonal Hermitian form with matrix $J = \begin{smatrix} & &&1\\&&1\\&-1\\-1\end{smatrix}$; and let $H=\GL_2\times_{\GL_1}\GL_2$. We consider $H$ as a subgroup of $G$ via
  \[ \iota:
   \left[\begin{pmatrix} a & b\\ c & d\end{pmatrix}, \begin{pmatrix} a' & b'\\ c' & d'\end{pmatrix}\right]
   \mapsto \begin{smatrix} a &&& b\\ & a' & b' & \\ & c' & d' & \\ c &&& d\end{smatrix}.
  \]
  We let $B_G$ be the Borel subgroup of upper-triangular matrices in $G$, and $B_H = H \cap B_G$. We let $T$ be the diagonal torus of $G$, and for integers $r_1, r_2, c$ with $c = r_1 + r_2 \bmod 2$, we write $(r_1, r_2; c)$ for the character of $T$ mapping $\begin{smatrix}tx\\ &ty \\ &&ty^{-1}\\&&&tx^{-1}\end{smatrix}$ to $x^{r_1} y^{r_2} t^c$. Thus $(r_1, r_2; c)$ is dominant if $r_1 \ge r_2 \ge 0$, and the half-sum $\rho_G$ is equal to $(2, 1; 0)$.

 \subsection{Dirichlet characters}

  If $\chi$ is a Dirichlet character, we let
  \[ G(\chi) = \sum_{a \in (\ZZ/N\ZZ)^\times} \chi(a) \exp(2\pi i a / N)\]
  be the Gauss sum, where $N$ is the conductor of $\chi$. We formally define $G(\chi) = 1$ if $\chi$ is the trivial character.

  For a prime $\ell \mid N$, we let $\chi = \chi_\ell \cdot \chi^{(\ell)}$ be the unique factorisation of $\chi$ as a product of Dirichlet characters of $\ell$-power and prime-to-$\ell$ conductors. We note the formula
  \[
   G(\chi) = \prod_{\ell \mid N} \left(G(\chi_{\ell}) \cdot \chi^{(\ell)}(\ell^{v_\ell(N)})\right).
  \]

  We write $\widehat\chi$ for the unique character of $\QQ^\times \backslash \AA^\times$ such that $\widehat\chi(\varpi_\ell) = \chi(\ell)$ for $\ell \nmid N$, where $\varpi_\ell$ is a uniformizer at $\ell$. We caution the reader that the restriction of $\widehat\chi$ to $\widehat{\ZZ}^\times \subset \AA^\times$ is the \emph{inverse} of $\chi$, where we identify $\chi$ with a character of $\widehat{\ZZ}^\times$ via the natural map $\widehat\ZZ \to \ZZ/N$. Then we have
  \[
   G(\chi_{\ell}) \cdot \chi^{(\ell)}(\ell^{v_\ell(N)})
   = \int_{\ell^{-c} \ZZ_\ell^\times} \widehat{\chi}^{-1}(x) \psi_\ell(x) \mathrm{d}x
   = \varepsilon(\widehat\chi_\ell, \psi_\ell, \mathrm{d}x),
  \]
  where $c = v_\ell(N_\chi)$, and the local $\varepsilon$-factor is as defined in \cite{deligne73}. Thus $G(\chi) = \prod_{\ell} \varepsilon(\widehat\chi_\ell, \psi_\ell, \mathrm{d}x)$. Here $\psi$ is the unique additive character of $\AA / \QQ$ with $\psi_\infty(x_\infty) = \exp(-2\pi i x)$, and $\mathrm{d}x$ the unramified Haar measure.

\section{Cyclotomic \texorpdfstring{$p$}{p}-adic \texorpdfstring{$L$}{L}-functions}
\label{sect:cycloL}

 We begin by recalling, and somewhat extending, results from \cite{LPSZ1} on the existence of one-variable cyclotomic $p$-adic $L$-functions. These will be used to give a precise formulation of the Iwasawa main conjecture.

 \subsection{Automorphic representations of $\GSp_4$}

  Let $\pi$ be an automorphic representation of $G$ which is globally generic, and such that $\pi_\infty$ is a discrete-series or limit of discrete-series representation with non-vanishing $(\mathfrak{p}, K)$-cohomology. Then there are integers $r_1 \ge r_2 \ge -1$ such that $\pi_\infty$ has infinitesimal character $(r_1 + 2, r_2 + 1; c)$ for some $c$, and via twisting by a power of the norm character we may suppose $c = -(r_1 + r_2)$. We write $\Pi$ for the unitary twist of $\pi$; and we write $\chi_{\pi}$ for the nebentype of $\pi$, i.e.~the Dirichlet character such that $\pi$ has central character $|\cdot|^{-(r_1 + r_2)} \widehat{\chi}_\pi$. Note that $\chi_{\pi}(-1) = (-1)^{r_1 + r_2}$.

  \begin{remark}
   To compare our parameters with the more classical language: if $\pi$ is not a Yoshida lift, then there exists an automorphic representation which is locally isomorphic to $\pi$ at all finite primes, and is generated by a holomorphic Siegel modular form of weight $(r_1 + 3, r_2 + 3)$ (thus scalar-valued if $r_1 = r_2$, and vector-valued if $r_1 > r_2$).
  \end{remark}

  The Archimedean factor $\pi_\infty$ is discrete-series if $r_2 \ge 0$, in which case $\pi$ has non-vanishing $(\mathfrak{g}, K)$-cohomology with coefficients in the algebraic representation of weight $\nu = (r_1, r_2; r_1 + r_2)$. In the $r_2 = -1$ case, it is a non-regular limit of discrete series (and $\nu$ is a non-dominant weight), so it does not contribute to $(\mathfrak{g}, K)$-cohomology, but it does contribute to $(\mathfrak{p}, K)$ cohomology, and thus to coherent cohomology of toroidal compactifications of the Shimura variety for $G$.

  \begin{notation}
   For $\ell$ a prime such that $\pi_\ell$ is unramified, let $P_\ell(\pi, X)$ be the polynomial such that
   \[ L(\Pi_\ell, s - \tfrac{r_1+r_2+3}{2}) = L(\pi_\ell, s - \tfrac{3}{2}) = P_\ell(\pi, \ell^{-s})^{-1}.\]
  \end{notation}

  \begin{theorem}
   Let $S$ be any finite set of places containing $\infty$ and all primes such that $\pi_\ell$ is ramified. Then the coefficients of the $P_\ell(\pi, \ell^{-s})$ for $\ell \notin S$ generate a finite extension $\QQ(\pi) / \QQ$ (independent of the choice of $S$), and the finite part $\pi_{\mathrm{f}}$ is definable as a $\QQ(\pi)$-linear representation.
  \end{theorem}

  \begin{proof}
   This follows from Arthur's classification of discrete automorphic representations of $\GSp_4$ (see \cite{arthur04, geetaibi18}), together with the strong-multiplicity-one theorems for $\GL_n$ for $n = 1,2,4$.
  \end{proof}

  As a by-product of the construction of $p$-adic $L$-functions in \cite{LPSZ1}, one can obtain the following rationality result for $L$-values. To include the case $r_1 = r_2$ we need to impose the following hypothesis (compare Hypothesis 10.5 of \cite{LPSZ1}):

  \begin{hypothesis}
   For $\ep \in \{ \pm 1\}$, we let ``Hypothesis $NV_{\ep}(\pi)$'' be the following assumption:
   \begin{itemize}
    \item There exists a Dirichlet character $\chi_{\ep}$ such that $\chi_{\ep}(-1) = \ep$ and $L(\Pi \times \chi_{\ep}, \tfrac{1}{2}) \ne 0$.
   \end{itemize}
   We let $NV^\sharp_{\ep}(\pi)$ be the slightly stronger hypothesis that we may choose such a character with  $\chi_{\ep} \notin \{1, \chi_{\pi}^{-1}\}$ (this will not be used until \cref{sect:ESclass} below).
  \end{hypothesis}

  \begin{remark}
   We expect that all (limit of) discrete-series automorphic representations of $\GSp_4$ with $r_1 = r_2$ should satisfy Hypothesis $NV_{\ep}(\pi)$ for either sign $\ep$, analogously to the results of Rohrlich \cite{rohrlich89} for $\GL_2$. However, this seems to be rather difficult in general.

   As noted in the introduction, if $\pi$ is a $\theta$-lift from $\GL_2$ over an imaginary quadratic field $K$, then $\pi \cong \pi \times \chi_K$, where $\chi_K$ is the odd quadratic character associated to $K$. So for these representations, the hypotheses $NV_{\ep}(\pi)$ and $NV_{-\ep}(\pi)$ are equivalent to each other; in particular, if $L(\Pi, \tfrac{1}{2}) \ne 0$ then both hypotheses hold.
  \end{remark}

  \begin{theorem}[Rationality for $\GSp_4$]
   \label{thm:rationality}
   Let $\ep \in \{ \pm 1\}$, and assume that either $r_1 > r_2$, or Hypothesis $NV_{-\ep}(\pi)$ holds. Then we may find a period $\Omega_\pi^{\ep} \in \CC^\times$, well-defined up to multiplication by $\QQ(\pi)^\times$, with the following property: for every integer $j$ with $0 \le j \le r_1 - r_2$, and every Dirichlet character $\chi$ such that $(-1)^j \chi(-1) = \ep$, we have
   \[
    \frac{\Lambda(\Pi \times \chi, \tfrac{1-r_1+r_2}{2} + j)}{G(\chi)^2 \cdot  \Omega_\pi^{\ep}} \in \QQ(\pi, \chi).
   \]
  \end{theorem}

  \begin{remark}
   In the regular-weight ($r_2 \ge 0$) case, this result follows from a more general theorem due to Grobner--Raghuram \cite{grobnerraghuram14}, whose proof relies on Betti cohomology of $\GL_4$ symmetric spaces. In this approach, the auxiliary condition $NV_{-\ep}(\pi)$ when $r_1 = r_2$ is not needed, so one can still obtain algebraicity of the sign $\ep$ twists even in the (improbable) event that the sign $-\ep$ twists all vanish.

   However, these methods cannot be applied when $r_2 = -1$, since in this case $\pi$ (and, more relevantly, its lifting to $\GL_4$) do not contribute to Betti cohomology. In this case the above theorem appears to be new.
  \end{remark}

 \subsection{Hecke parameters at $p$}

  We choose a prime $p$ and an embedding $\iota: \QQ(\pi) \into L$, where $L / \Qp$ is a finite extension. We suppose that $\pi$ is unramified at $p$, so we may write
  \[ P_p(\pi, X) = (1 - \alpha X)(1 - \beta X)(1 - \gamma X)(1 - \delta X),\qquad \alpha\delta = \beta\gamma = p^{(r_1 + r_2 + 3)} \chi_\pi(p) \]
  We also suppose that $\pi$ is Klingen-ordinary (with respect to $\iota$), so we can (and do) order the parameters $(\alpha, \dots, \delta)$ so that $v_p(\alpha\beta) = r_2 + 1$.

  \begin{proposition}
   \label{prop:noexzero}
   None of $(\alpha, \dots, \delta)$ have the form $p^r \zeta$, for $r \in \ZZ$ and $\zeta$ a root of unity.
  \end{proposition}

  \begin{proof}
   This follows by comparing the complex and $p$-adic absolute values of the Hecke parameters. More precisely, the analytically-normalised parameters $p^{-(r_1 + r_2 + 3)/2} \alpha$ etc, must have complex absolute value strictly between $p^{-1/2}$ and $p^{1/2}$, since otherwise $\Pi_p$ would be non-unitary. (Of course we expect that their absolute value is exactly 1, i.e.~that $\Pi_p$ is tempered, and this is known if $r_2 \ge 0$; but the weaker bound is known for $r_2 = -1$ as well, and this suffices for our purposes.) Hence none of $\alpha, \dots, \delta$ can be of the form $p^r\zeta$, except possibly if $r_1 + r_2$ is odd and the ``bad'' Hecke parameter is a root of unity times $p^{(r_1 + r_2 + 3)/2}$. However, Klingen-ordinarity implies that the $p$-adic valuation of any Hecke parameter must be either $\le r_2 + 1$ or $\ge r_1 + 2$, so it cannot be equal to $(r_1 + r_2 + 3)/2$.
  \end{proof}

 \subsection{Cyclotomic $p$-adic $L$-functions}

   \begin{definition}
    For $a \in \ZZ$, and $\rho$ a Dirichlet character of conductor $p^r$, we define
    \[ R_p(\pi,\rho, a) =
     \left(1 - \frac{p^{a+r_2 + 1}}{\alpha} \right) \left(1 - \frac{p^{a+r_2 + 1}}{\beta}\right) \left(1 - \frac{\gamma}{p^{a + r_2 + 2}} \right) \left(\frac{\delta}{p^{a + r_2 + 2}} \right)
    \]
    if $\rho$ is trivial, and
    \[ R_p(\pi, \rho, a) = \left(\frac{p^{(2a + 2r_2 + 4)}}{\alpha \beta}\right)^r\]
    if $\rho$ is non-trivial.
   \end{definition}

  Note that $R_p(\pi, \rho, a)$ is always non-zero. This is obvious if $\rho$ is nontrivial; if $\rho$ is trivial, it follows from Proposition \ref{prop:noexzero}. The main result of \cite{LPSZ1} was the following. Let $\Gamma = \Zp^\times$, and $\Lambda_L(\Gamma)$ its Iwasawa algebra with coefficients in $L$.

  \begin{theorem}[{\cite[Theorem A]{LPSZ1}}]
   \label{thm:cycloL}

   Suppose $r_2 \ge 1$ and $\pi$ is unramified and Klingen-ordinary at $p$. Let $\ep \in \{ \pm 1\}$, and if $r_1 = r_2$, assume that $NV_{-\ep}(\pi)$ holds. Then, for any period $\Omega_\pi^{\ep}$ as in \cref{thm:rationality}, there exists a $p$-adic $L$-function $\cL_p^{\ep}(\pi) \in \Lambda_L(\Gamma)^{\ep}$ such that
   \[
    \cL_p^\ep(\pi)(j + \rho) = R_p(\pi, \rho, j) \cdot \frac{\Lambda(\Pi \times \rho^{-1}, \tfrac{1-r_1+r_2}{2} + j)}{G(\rho^{-1})^2 \cdot  \Omega_\pi^{\ep}}
   \]
   for all integers $0 \le j \le r_1-r_2$ and Dirichlet characters $\rho$ of $p$-power conductor, with $(-1)^j\rho(-1) = \ep$.
  \end{theorem}

  Again, there is an alternative approach via $\GL_4$ Betti cohomology, described in \cite{DJR20}; this does not require the assumption $NV_{-\ep}(\pi)$ in the parallel-weight case, but the interpolation property proved in \emph{op.cit.} is a little less complete (the Archimedean zeta-integral is not evaluated, and the local factor $R_p(\pi, \rho, j)$ is not computed when $\rho$ is trivial).

  \begin{proposition}
   The above theorem remains valid for all $r_2 \ge -1$.
  \end{proposition}

  \begin{proof}
   The assumption $r_2 \ge 1$ is used in an essential way in the proof of Theorem 3.6 of \cite{LPSZ1}. However, in the case $r_2 = -1$, we can substitute instead Theorem 1.1(1) of \cite{pilloni20}; this gives rise to a perfect complex of modules supported in degrees $\{0, 1\}$, rather than a single module, but this is sufficient for the applications in \emph{op.cit.}~(since the compatibility with base-change is not used). The arguments used to prove this also apply for $r_2 = 0$ (Pilloni, pers.comm.). We also need here the fact that Moriyama's formula \cite{moriyama04} for the archimedean Whittaker function is also valid for generic limits of discrete series, so the evaluation of the local zeta-integral at $\infty$ goes through without change.
  \end{proof}

  \begin{remark}
   As a corollary of the above, one sees that for $r_2 = -1$, and $r_1 \ge 1$ odd, we have $L(\Pi \times \chi, \tfrac{1}{2}) \ne 0$ for all but finitely many Dirichlet characters $\chi$ of $p$-power conductor. This is a non-regular-weight analogue of the non-vanishing results of \cite[Theorem A]{DJR20}.
  \end{remark}

 \subsection{Equivariant $L$-functions and normalised periods}

  We suppose henceforth that $p > 2$.

  \begin{notation}
   Let $S$ be a finite set of primes including all primes of ramification of $\pi$, but not $p$. We let $\cR$ be the set of square-free integers coprime to $p S$, and for $m \in \cS$, we let $\Delta_m$ be the the maximal quotient of $(\ZZ/ m)^\times$ of $p$-power order.
  \end{notation}

  Then, under the same hypotheses as above, the construction of \cite{LPSZ1} extends to give \emph{equivariant} $p$-adic $L$-functions
  \[ \cL_p^\ep(\pi, \Delta_m) \in \Lambda_L(\Gamma)^{\ep} \otimes_L L[\Delta_m], \]
  compatible under the projection maps $L[\Delta_n] \to L[\Delta_m]$ for $m \mid n$, and such that the image of $\cL_p^\ep(\pi, \Delta_m)$ under a primitive character $\eta$ of $\Delta_m$ is exactly $\cL_p^{\ep}(\pi \times \eta^{-1})$. Here we define periods for twists of $\pi$ by
  \[ \Omega_{\pi \times \eta}^\ep = \Omega_{\pi}^{\ep} \prod_{\ell \mid m} G(\eta_\ell)^2,\]
  where we write $\eta = \prod_\ell \eta_\ell$ as a product of Dirichlet characters of prime conductor.

  \begin{proposition}
   We may choose $\Omega_\pi^{\ep}$ such that $\cL_p^\ep(\pi, \Delta_m) \in \Lambda_{\cO}(\Gamma) \otimes \cO[\Delta_m]$, for all $m$, where $\cO$ is the ring of integers of $L$.
  \end{proposition}

  \begin{proof}
   This follows from the finite generation of the cohomology modules of higher Hida theory.
  \end{proof}

  \begin{definition}
   We say the period $\Omega_\pi^{\ep}$ is \textbf{optimally integrally normalised} if the fractional ideal of $\cO$ generated by the coefficients of the equivariant $p$-adic $L$-functions $\cL_p^\ep(\pi, \Delta_m)$, for all square-free $m$ as above, is the unit ideal.
  \end{definition}

  Note that this determines $\Omega_\pi^{\ep}$ uniquely up to a $p$-adic unit, unless the $\cL_p^\ep(\pi, \Delta_m)$ are all zero (which is clearly impossible if $r_1 > r_2$, but cannot be \emph{a priori} ruled out in the parallel-weight case). In particular, the ideal of $\Lambda_{\cO}(\Gamma)^{\ep}$ generated by $\cL_p^\ep(\pi)$, for an optimally-normalised period $\Omega_{\pi}^{\ep}$, is well-defined whenever $\cL_p^\ep(\pi)$ exists.

 \subsection{Galois representations}

  Associated to $\pi$ and the prime $p$ (and the choice of embedding $\iota: \QQ(\pi) \into L \subset \overline{\QQ}_p$), there exists an $L$-linear Galois representation $V_p(\pi)$, such that
  \[ \det(1 - X \operatorname{Frob}_\ell^{-1}: V_p(\pi)) = P_\ell(\pi, X) \]
  for all but finitely many $\ell$, where $\operatorname{Frob}_\ell$ is an arithmetic Frobenius at $\ell$. (This is the representation denoted $\rho_{\pi, p}$ in \cite{LSZ17}, where the case $r_2 \ge 0$ is considered; for $r_2 = -1$ see Theorem 5.3.1 of \cite{pilloni20}). This characterises $V_p(\pi)$ uniquely up to semisimplification.

  It is expected that $V_p(\pi)$ be irreducible (for all $p$ and $\iota$) if $\pi$ is not a Yoshida lift from $\GL_2 \times \GL_2$; and this is known if $r_2 \ge 0$ and $p$ is large enough \cite{ramakrishnan13}.

  \begin{proposition}
   Suppose $\pi$ is Klingen-ordinary, and if $r_2 = -1$, suppose also that $\alpha \ne \beta$. Then $V_p(\pi)$ is crystalline at $p$ with Hodge--Tate weights $(0, -1-r_2, -2-r_1, -3-r_1-r_2)$, and we have
   \[ \det(1 - X \varphi: \Dcris(V_p(\pi))) = P_p(\pi, X). \]
   Moreover, there exists a unique 2-dimensional subrepresentation
   \[ \cF^2 V_p(\pi) \subset V_p(\pi)|_{\Gal(\overline{\QQ}_p / \Qp)}, \]
   whose Hodge--Tate weights are $0$ and $-1-r_2$, and whose crystalline $\varphi$-eigenvalues are $\alpha$ and $\beta$.
  \end{proposition}

  \begin{proof}
   For $r_2 \ge 0$ this is shown in \cite{urban05}. The non-regular weight case is treated in \cite{boxerpilloni20}, extending earlier results due to Jorza and Mok for representations lifted from $\GL_2$ over quadratic extensions.
  \end{proof}

 \subsection{Formulating the main conjecture}

  \begin{notation}
   Let $\QQ_\infty$ denote the extension $\QQ(\mu_{p^\infty})$, and similarly $\QQ_{p, \infty}$.
  \end{notation}

  We use the subspace $\cF^2 V_p(\pi)$ (or more precisely its orthogonal complement in the dual $V_p(\pi)^*$) as the local condition defining a Selmer complex $\widetilde{R\Gamma}\left(\QQ_\infty, V_p(\pi)^*(-1-r_2)\right)$ (cf.~\cite{nekovar06} and \cite[\S 11]{KLZ17}), with the local conditions at all primes $\ell \ne p$ being the unramified ones. The cohomology groups of this complex are finitely-generated $\Lambda_{L}(\Gamma)$-modules, and the fibre at $j \in \ZZ$ is a Selmer complex for $V_p(\pi)^*(-1-r_2-j)$ over $\QQ$. (The twist $-1-r_2$ corresponds to our choice of parameters for the analytic $p$-adic $L$-function.)

  Similarly, we may define an integral Selmer complex for a Galois-stable $\cO$-lattice $T_p(\pi)$ in $V_p(\pi)$; we shall only use this when $T_p(\pi)$ is residually irreducible, in which case it is unique up to scaling and hence the isomorphism class of the Selmer complex is canonically determined.

  \begin{conjecture}[Iwasawa main conjecture for $\pi$] \
   \begin{enumerate}[(i)]
   \item The group $\widetilde{H}^2(\QQ_\infty, V_p(\pi)^*(-1-r_2))$ is a torsion $\Lambda_L(\Gamma)$-module, and its characteristic ideal is generated by $\cL_p(\pi) = \cL_p^{+}(\pi) + \cL_p^-(\pi)$.
   \item If $T_p(\pi)$ is residually irreducible (for one, and hence every, lattice $T_p(\pi)$), then the characteristic ideal of $\widetilde{H}^2(\QQ_\infty, T_p(\pi)^*(-1-r_2))$ as a $\Lambda_{\cO}(\Gamma)$-module is generated by $ \cL_p(\pi)$, where the $p$-adic $L$-function is defined using an optimally-normalised period.
   \end{enumerate}
  \end{conjecture}

  We shall establish one inclusion in this conjecture below, under various auxiliary hypotheses; see Theorem \ref{thm:IMC}.

\section{Tame test data}
 \label{sect:testdata}

 \subsection{The idea} We would now like to study the variation of $p$-adic $L$-functions (and, subsequently, also Euler systems) in $p$-adic families. However, what we shall achieve is something a little weaker. Our $p$-adic $L$-function is defined as a value of a zeta-integral, which depends on certain auxiliary test data at the primes in $S$. If these data are ``optimal'' for $\pi$, then the resulting $p$-adic $L$-function interpolates the complex $L$-function (with the correct local factors at all places, including the ones in $S$). However, it seems to be somewhat difficult to show that one can choose test data which are simultaneously optimal for all specialisations of a given $p$-adic family.

 So we shall, instead, interpolate the zeta-integrals for \emph{arbitrary} choices of test data, thus sidestepping the question of whether the test data can be chosen so they are ``optimal'' for all specialisations at once.

 \subsection{Setup}

  \begin{notation} \
   \begin{itemize}
   \item Let $M_0, N_0$ be positive integers coprime to $p$, with $M_0^2 \mid N_0$, and $\chi_0$ a Dirichlet character of conductor $M_0$. We shall consider automorphic representations of $G$ with conductor $N_0$, and central character $\widehat\chi_0$ (up to powers of the norm character); the condition $M_0^2 \mid N_0$ is a necessary condition for these to exist.
   \item Fix Dirichlet characters $\chi_1, \chi_2$ of conductors $M_1, M_2$ (which we again suppose prime to $p$), satisfying the condition $\chi_0 \chi_1 \chi_2 = \mathrm{id}$.
   \item Let $S$ be a finite set of primes containing all those dividing $N_0 M_1 M_2$ (but not $p$).
   \end{itemize}
  \end{notation}

  Let $E$ be a field containing $\QQ(\chi_1, \chi_2)$. By \emph{tame test data} with coefficients in $E$, we shall mean a triple $\gamma_S = (\gamma_{0, S}, \Phi_{1,S},\Phi_{2,S})$, where:
  \begin{itemize}
  \item $\gamma_{0, S} \in E[G(\QQ_S)]$, where $\QQ_S = \prod_{\ell \in S} \Ql$;
  \item $\Phi_{i,S} \in \cS(\QQ_S^2, E)$ for $i=1,2$ which lies in the $\widehat\chi_i$-eigenspace for $\ZZ_S^\times$.
  \end{itemize}
  These test data are the pure tensors in a vector space
  \[ \cT_S(E) = E[G(\QQ_S)] \otimes \cS(\QQ_S^2, , \widehat{\chi}_{1, S}) \otimes \cS(\QQ_S^2, , \widehat{\chi}_{2, S})\]
  on which the group $(\GSp_4 \times \GL_2 \times \GL_2)(\QQ_S)$ acts.

  We let $K_S$ be the quasi-paramodular subgroup of $G(\QQ_S)$ of level $(N_0, M_0)$; and we let $\widehat{K}_S$ be some open compact subgroup of $G(\QQ_S)$ such that:
  \begin{itemize}
   \item $\widehat{K}_S \subseteq \gamma_{0, S} K_S \gamma_{0, S}^{-1}$,
   \item the intersection $\widehat{K}_S \cap H$ acts trivially on $\Phi_S$.
  \end{itemize}
  We define $K^p$ and $\widehat{K}^p$ to be the products of $K_S$ and $\widehat{K}_S$ with $G(\Af^{pS})$, and $\Phi_i^{(p)} = \Phi_{i,S} \cdot \operatorname{ch}\left((\widehat{\ZZ}^{S \cup \{p\}})^2\right)$.

 \subsection{The correction term $Z_S$}

  Let $\pi$ be a cuspidal automorphic representation of $G$, with Archimedean component given by a pair of integers $r_1 \ge r_2 \ge -1$ as before; and $\Pi$ its unitary twist. We suppose $\pi$ has conductor $N_0$ and nebentypus $\chi_0$.

  \begin{definition}
   Let $\ell \in S$. For $W_0 \in \cW(\pi_\ell)$, $\Phi_{1,\ell},\Phi_{2,\ell} \in \cS(\Ql^2, \CC)$, and $\rho$ a Dirichlet character of conductor prime to $S$, we consider the normalised zeta-integral
   \begin{multline*}
    \mathfrak{z}(W_0, \Phi_{1,\ell}, \Phi_{2,\ell}, \rho, s_1,s_2) =
    \frac{1}{L(\Pi_\ell \times \rho_\ell^{-1}, s_1 + s_2 - \tfrac{1}{2}) L(\Pi_\ell \times \chi_{2, \ell}, s_1 - s_2 + \tfrac{1}{2})}
    \times \\ \int_{(Z_G N_H \backslash H)(\Ql)} \widehat\rho_\ell(\det h)^{-1} |\det h|^{(r_1 + r_2)/2}W_0(h) f^{\Phi_1}(h_1; \widehat\chi_{1, \ell}^{-1} \widehat\rho_\ell^{-1}, s_1) W^{\Phi_2}(h_2;\widehat\chi_{2, \ell}^{-1}\widehat\rho_\ell^{-1},s_2) \, \mathrm{d}h
   \end{multline*}
  \end{definition}

  \begin{remark}
       In our applications, $\rho$ will be of $p$-power conductor.
  \end{remark}

  We note that by Theorem 8.8(i) of \cite{LPSZ1} (applied to $\Pi_\ell \times \widehat{\rho}_\ell^{-1}$), the values of $\mathfrak{z}(-)$ are polynomials in $\ell^{\pm s_i}$, and there is no $(s_1, s_2) \in \CC$ where these polynomials vanish for all $(W_0, \Phi_{1, \ell}, \Phi_{2, \ell})$.

  Since $\pif$ is generic, it is quasi-paramodular in the sense of \cite{okazaki}, so its Whittaker model is generated by a canonical vector $W_0^{\new} \in \cW(\pif)$ (normalised to take the value 1 at the identity). We shall be interested in the case when $s_1 = \tfrac{j+1}{2}$, $s_2 = \tfrac{1-r_1+r_2+j}{2}$ for an integer $j$, and $W_0 = \gamma_{0, \ell} W_0^{\new}$ for some $\gamma_{0, \ell} \in \CC[G(\Ql)]$, so the test data is determined by the triple
  \[ \gamma_\ell = (\gamma_{0, \ell}, \Phi_{1, \ell}, \Phi_{2, \ell}).\]

  \begin{proposition}
   \label{prop:testdata}
   There exists a polynomial $Z_\ell(\pi_\ell, \gamma_\ell) \in \CC[X, X^{-1}]$ such that we have
   \[ G(\chi_{2, \ell})^2 \mathfrak{z}(W_0, \Phi_{1,\ell}, \Phi_{2,\ell}; \rho, \tfrac{j+1}{2}, \tfrac{1-r_1+r_2+j}{2}) = Z_\ell(\pi, \gamma_\ell)(\ell^{-j} \rho(\ell)^{-1}), \]
   for all $j\in\CC$ and $\rho$ unramified at $\ell$; and if $E$ contains $\QQ(\pif, \chi_1, \chi_2)$, then the map
   \[ \cT_S(\CC) \to \CC[X, X^{-1}],\qquad \gamma_\ell \mapsto Z_\ell(\pi_\ell, \gamma_\ell)\]
   is the base-extension to $\CC$ of an $E$-linear map $\cT_S(E) \to E[X, X^{-1}]$. Moreover, this map is surjective.
  \end{proposition}

  \begin{proof}
   Since $\rho$ is unramified, there exists some $\sigma \in \CC$ such that $\widehat\rho_\ell = |\cdot|^{\sigma}$ as characters of $\Ql^\times$; hence, replacing $s$ with $s - \sigma$, we can assume $\rho = 1$. The existence of the polynomial $Z_\ell(\pi_\ell, \gamma_\ell)$ is now clear from the results of \cite[\S 8]{LPSZ1} cited above.

   We now check that the map is defined over $E$. The $E[G(\Ql)]$-orbit of $W_0^{\new}$ is precisely the $E$-rational Whittaker model $\cW(\pi_\ell)_E$, as defined in \cite[\S 8.3.1]{LPSZ1}. From the equivariance property of vectors in this model, one checks easily that for all $\sigma \in \operatorname{Aut}(\CC / E)$ we have
   \[ Z_\ell(\pi_\ell, \gamma_\ell)^{\sigma} = \widehat{\chi}_{2, \ell}(\kappa_\ell(\sigma))^{-2} Z_\ell(\pi_\ell, \gamma_\ell) = \chi_{2, \ell}(\kappa_\ell(\sigma))^{2} Z_\ell(\pi_\ell, \gamma_\ell),\]
   where we write $\chi_2 = \prod_{\ell \in S} \chi_{2, \ell}$ as a product of classical Dirichlet characters of prime-power conductors, and $\kappa_\ell$ denotes the $\ell$-adic cyclotomic character. Including the Gauss sum corrects for the $\kappa_\ell(\sigma)$ term, so the map is defined over $E$. Finally, the image of this map is clearly an ideal of $E[X, X^{-1}]$, and it cannot vanish identically for any value of $X$; hence it is the whole ring.
  \end{proof}

  \begin{notation}
   For a $p$-adic field $L$ and embedding $E \into L$ as before, and an integer $m \in \cR$, we regard $Z_\ell(\pi_\ell, \gamma_\ell)$ as an element of $\Lambda_L(\Gamma) \otimes L[\Delta_m]$, via evaluation at $X = \ell^{-\mathbf{j}} \cdot [\ell^{-1} \bmod m]$. If $\gamma_S = (\gamma_\ell)_{\ell \in S}$, we set
   \[ Z_S(\pi_S, \gamma_S) = \prod_{\ell \in S} Z_\ell(\pi_\ell, \gamma_\ell) \in \Lambda_L(\Gamma)\otimes L[\Delta_m].\]
   This depends implicitly on the pair $\uchi = (\chi_1, \chi_2)$ and we shall write $Z_S(\pi_S, \uchi, \gamma_S)$ where necessary to clarify this.
  \end{notation}

  \begin{remark}
   We say the test datum $\gamma_S$ is ``optimal for $\pi$'' if the factor $Z_S$ is identically 1. The surjectivity statement of Proposition \ref{prop:testdata} implies that, for any given $\pi$, optimal test data for $\pi$ exist. However, it is much less clear whether, given a $p$-adic family of representations $\upi$ as in the next section, one can choose a $\gamma_S$ which is simultaneously optimal for all classical specialisations of $\upi$. If true, this would simplify our lives considerably!

   (Using the work of Cheng \cite{cheng21}, who has evaluated the $\GSp_4 \times \GL_2$ zeta-integrals explicitly when $W_0$ is the paramodular new-vector, one can choose such a ``uniformly optimal'' test datum at $\ell$ for any prime $\ell \nmid M_1 M_2$. However, this is not sufficient for our purposes, since we need to be able to take the $\chi_i$ non-trivial.)
  \end{remark}

\section{Families of coherent cohomology classes}

 In this section, we discuss the local geometry of the $\GSp(4)$-eigenvariety, both at regular and at singular weights.

 \subsection{P-adic families for $G$}

  Let $\cW$ be the rigid-analytic weight space parametrising characters of $\Zp^\times$ (i.e.~the rigid-analytic generic fibre of $\operatorname{Spf} \Lambda(\Zp^\times)$). The space $\cW^2$ thus parametrises pairs of characters of $\Zp^\times$; we let $\br_1, \br_2$ be the two universal characters valued in $\cO(\cW^2)^\times$. Let $\pmb{\nu}$ be the $\cO(\cW^2)^\times$-valued character $(\br_1, \br_2; \br_1+\br_2)$ of $T(\Zp)$, and for any map of rigid spaces $U \to \cW^2$ with $U$ affinoid, let $\nu_U$ be the pullback of $\pmb{\nu}$ to $\cO(U)^\times$. Finally, let $\mathbb{T}^-$ be the Hecke algebra generated by the spherical Hecke operators away from $S \cup \{p\}$, and the ``anti-dominant'' Hecke operators at $p$.

  From the results of \cite[\S 6.8]{boxerpilloni20}, we know that there exists a rigid space $\cE \xrightarrow{\kappa} \cW^2$, with a map $\mathbb{T}^- \to \cO(\cE)$ (the eigenvariety for $G$), and graded coherent sheaves $\cH^{k}_{\cusp, w_j}= H^k(\cM^{\bullet,-,\fs}_{\cusp, w_j})$ on $\cE$ for $0 \le j, k \le 3$, whose pushforward to $U$ along $\kappa$ are the locally-analytic overcovergent cohomology groups $H^k_{w_j,\an}(K^p, \nu_U, \cusp)^{(-,\fs)}$. By construction, the points of $\cE$ biject with systems of $\mathbb{T}^-$-eigenvalues appearing in one of these modules.

  \begin{proposition}
   Let $\pi$ is a $p$-stabilised automorphic representation of weight $\nu = (r_1, r_2; r_1 + r_2)$ with $r_1 \ge r_2 \ge -1$, as in \S\ref{sect:cycloL} above, and unramified at $p$; and suppose its Galois representation $V_p(\pi)$ is irreducible. Then $\pi$ defines a point $P \in \cE$ with $\kappa(P) = \nu$; and there is a neighbourhood of $P$ in which the sheaves $\cH^{k}_{\cusp, w_j}$ are zero for $j \ne 3-k$, and free as $\kappa^{-1}\left(\cO(\cW^2)\right)$-modules for $j = 3-k$. Moreover, these sheaves are generically of rank 1 over $\cO(\cE)$ (although they may not be free).
  \end{proposition}

  \begin{proof}
   We know that the cohomology $H^k_{w_j,\an}(K^p, \nu_U, \cusp)^{(-,\fs)}$ is concentrated in degrees $[0, 3-j]$, and the non-cuspidal version $H^k_{w_j,\an}(K^p, \nu_U)^{(-,\fs)}$ is concentrated in degrees $[3-j, 3]$. Since the Galois representation associated to $\pi$ is irreducible, and the Hecke eigensystems appearing in the boundary cohomology all come from reducible Galois representations, the map from cuspidal to non-cuspidal cohomology becomes a quasi-isomorphism after localising at the $\pi$-eigenspace; thus the cohomology is concentrated in degree $k = 3-j$ alone.

   The same argument also shows that the $\pi$-eigenspace in the cohomology at weight $\nu$ (with coefficients in $L$ rather than in $\cO(U)$) is concentrated in degree $3-j$ only. Using the Tor spectral sequence relating these groups, and the local criterion for flatness in terms of $\operatorname{Tor}^1$, one deduces that the sheaves $\cH^{k}_{\cusp, w_{3-k}}$ are locally free over $\kappa^{-1}\cO(U)$. The generic rank statement follows by applying the classicity theorems of higher Coleman theory at points $P'$ which are $p$-adically close to $P$ and have sufficiently large weight, as in \cite{AIP15}.\footnote{The argument is given in \emph{op.cit.} under very strong local assumptions away from $p$, which serve to ensure multiplicity one of the relevant classical automorphic representations; but this can be relaxed, since the endoscopic classification of automorphic representations for $\GSp_4$ is now known.}
  \end{proof}

  Note that if the weight map $\kappa$ is \'etale at $P$, then $\kappa^{-1}(\cO(U))$ is locally isomorphic to $\cO(\cE)$, so the sheaves $\cH^{k}_{\cusp, w_{3-k}}$ are locally free over $\cE$ around $P$. This applies, in particular, to all Borel-ordinary $\pi$ whose weight is regular (using the enhanced slope bounds for interior cohomology proved in the final section of \cite{boxerpilloni20}). However, in singular weight, this does not apply; and indeed one knows that the $\GL_2$ eigencurve is not always \'etale over weight space at weight 1 classical points.

 \subsection{One-parameter families}

  In order to simplify the geometry of the situation (in particular for non-regular weights), we pass from the full eigenvariety $\cE$ to a codimension 1 subspace, in order that we can imitate the ideas developed by Bella\"iche in his study of $p$-adic $L$-functions for critical-slope elliptic modular forms \cite{bellaiche11a}. We fix an algebraic, but not necessarily regular, weight $\nu_0$, and an integer $n \ge 2$; and we let $\cW^{\flat} \subset \cW^2$ denote the space of weights of the form
  \[ \nu = \nu_0 + (n\kappa, \kappa)\]
  for $\kappa$ varying over $\cW$. We write $\cE^\flat = \kappa^{-1}(\cW^\flat) \subset \cE$. By construction this subspace contains a Zariski-dense set of regular algebraic weights, and we may choose these to be arbitrarily far from the walls of the Weyl chamber.

  \begin{proposition}
   Let $\pi_0$ be a $p$-stabilised automorphic representation of weight $\nu_0$, level $N_0$ and character $\chi_0$, which is Borel-ordinary at $p$. If $\cE^\flat$ is smooth at the point $P_0$ corresponding to $\pi_0$, then the restriction of $\cH^{k}_{\cusp, w_{3-k}}$ to $\cE^\flat$ is free of rank 1 around $P_0$.
  \end{proposition}

  \begin{proof}
   This is an instance of ``Bella\"iche's freeness lemma'' (Lemma 4.1 of \cite{bellaiche11a}).
  \end{proof}

  \begin{remark}
   If $\cE$ is \'etale over $\cW^2$ at $P_0$ (so the map induced by $\kappa$ on the tangent space at $P_0$ has full rank), then the smoothness holds for any choice of the integer $n$. If $\cE$ is smooth but not \'etale, and the map induced by $\kappa$ on tangent spaces at $P$ has rank 1, then this smoothness holds as long as the image of the map does not coincide with the tangent space of $\cW^\flat$, which is true for all $n$ except possibly for one ``bad'' value -- this is the motivation for allowing a general $n$. (It will always be false if the map on tangent spaces has rank 0, but we expect this case never to occur.)
  \end{remark}

  We assume henceforth that $\cE^\flat$ is smooth at $P_0$. We choose an affinoid $U \subset \cW^\flat$ containing $\nu_0$, and a neighbourhood $\tilde{U} \subset \cE^\flat$ of $P_0$ which is finite of degree $e$ over $U$ and unramified away from $P_0$, such that $\cH^{2}_{\cusp, w_{1}}$ and $\cH^{3}_{\cusp, w_{0}}$ are free of rank 1 over $\tilde{U}$.

  \begin{notation}
   We write $\upi$ for the projection to $\cO(\tilde{U})$ of the homomorphism $\mathbb{T}^- \to \cO(\cE)$; we can interpret this as a ``$p$-adic family of automorphic representations'' over $\tilde{U}$.
  \end{notation}

  The notation is justified by the fact that the specialisations of $\upi$ at suitable points $P \in \tilde{U}$ (with $\kappa(P)$ algebraic) are the eigenvalue systems associated to automorphic representations $\pi_P$ of weight $\kappa(P)$. We shall say a point $P \in \tilde{U}$ is \emph{good for $\upi$} if the representation $\pi_P$ exists and has level coprime to $p$. After possibly shrinking $\tilde{U}$ (to eliminate ramified but Iwahori-spherical representations), we can arrange that all points $P \in \tilde{U}$ whose weight is regular are good for $\upi$. However, there may also be some points of singular weight which are good for $\upi$ as well. In particular, $P_0$ itself is always good for $\upi$, by definition.

 \subsection{Specialisation at singular weights}

  The scheme-theoretic fibre of $\kappa$ at $P_0$ has the form $L[X] / X^e$, for some $e \ge 1$ (with $e = 1$ if the eigenvariety is \'etale); and the restriction of the sheaf $\cH^{k}_{\cusp, w_{3-k}}$ to this fibre maps isomorphically to the $\pi_0$-generalised eigenspace in the overconvergent cohomology at weight $\nu_0$. (We do not need to distinguish here between overconvergent cohomology $H^k_{w_j}$ and locally-analytic cohomology $H^k_{w_j, \an}$, since the map between the two is an isomorphism on the ordinary eigenspace.)

  \begin{proposition}
   There is an exact sequence
   \[ 0 \to H^2\{\pi_0\} \to H^2_{w_1}\{\pi_0\}  \to  H^3_{w_0}\{\pi_0\} \to H^3\{\pi_0\} \to 0,\]
   where $H^i\{\pi_0\}$ denotes the $\pi_0$-generalized eigenspace in the classical cohomology.
  \end{proposition}

  \begin{proof}
   This follows from the theory of the ``Cousin complex'' \cite{boxerpilloni20}, a complex whose terms are the small-slope parts of the overconvergent cohomology groups (this is defined for interior cohomology in \cite{boxerpilloni20}, but the $\pi_0$-parts of cuspidal and interior cohomology coincide in our case, since we have assumed it corresponds to an irreducible Galois representation). In our setting, the slope 0 part of the Cousin complex vanishes in degrees 0 and 1, so only $H^2$ and $H^3$ contribute.
  \end{proof}

  \begin{definition}
   We let $\uet$ be a generator of the rank 1 $\cO(\tilde{U})$-submodule of $\cH^{2}_{\cusp, w_{1}} |_{\tilde{U}}$ consisting of elements whose specialisation at $P_0$ vanishes to order $e-1$.
  \end{definition}

  These are precisely the classes whose image in $H^2_{w_1}\{\pi_0\}$ lies in the image of the one-dimensional $(\mathbb{T}^- = \pi_0)$-eigenspace in the classical cohomology (not just in the generalised eigenspace, which may be larger if the Hecke polynomial has a repeated root). So the specialisation of $\uet$ is a basis of this eigenspace.

  \begin{remark}
   All of the above discussion simplifies considerably if $\pi_0$ has regular weight, since in this case $e$ is always 1.
  \end{remark}

 \section{Construction of the $p$-adic $L$-function}\label{sect:3varpadicLfct}

  In this section, we explain the construction of a $3$-variable $p$-adic $L$-function, in which the weight variables $(r_1, r_2)$ and the cyclotomic variable all vary, via an adaptation of the strategy of \cite[\S 10]{LZ21-erl}. As before, let $L$ be a finite extension of $\Qp$ with ring of integers $\cO$.

  \label{def:deformable}
  We persist with the notation and assumptions of the previous section. Thus $\pi_0$ is an automorphic representation of $\GSp_4$ satisfying the following conditions:
  \begin{itemize}
   \item $\pi_0$ has weight $\nu_0 = (r_1, r_2; r_1 + r_2)$, with $r_1 \ge r_2 \ge -1$;
   \item $\pi_0$ has paramodular level $N_0$ (coprime to $p$) and nebentype character $\chi_0$;
   \item $\pi_0$ is Borel-ordinary\footnote{This could be relaxed somewhat -- it would suffice for the slope to be sufficiently small relative to $r_1$ and $r_2$. However, Borel-ordinarity is essential for the Euler-system arguments of \cref{sect:leadingterm}.};
   \item the 4-dimensional Galois representation $V_p(\pi_0)$ is irreducible;
   \item the $\GSp_4$ eigenvariety $\cE$ is smooth at the point corresponding to $P_0$;
   \item the differential of the weight map at $P_0$ is non-zero, so $\cE^\flat$ is smooth at $P_0$ (for a suitable choice of $n$ which we fix henceforth).
  \end{itemize}
  For simplicity we shall say that $\pi$ is \emph{deformable} if the last two conditions are satisfied. Note that we do \emph{not} suppose that $\cE$ is etale over weight space at $\pi_0$, since computations suggest that this will never be satisfied when $\pi_0$ is a theta-lift from $\GL_2$ over an imaginary quadratic field; while we expect that the weaker assumptions above should hold (at least if the roots of the Hecke polynomial are distinct, which is automatic in regular weights).

  \subsection{Families of Eisenstein series}

   We refer to \cite[\S 7]{LPSZ1} for the construction of $p$-adic families of $\GL_2$ Eisenstein series $\cE^{\Phi^{(p)}}(\kappa_1, \kappa_2; \chi^{(p)})$, depending on a prime-to-$p$ Schwartz function $\Phi^{(p)}$ and prime-to-$p$ Dirichlet character $\chi^{(p)}$ (both valued in $L$) and a pair of characters $\kappa_1, \kappa_2$ of $\Zp^\times$ (valued in some $p$-adically complete $L$-algebra $A$).

   \begin{note}
    Note that this Eisenstein series is $p$-depleted, i.e.~lies in the kernel of $U_p$; and it is zero on any components of $\Spec(A)$ which do not satisfy the parity condition $\kappa_1(-1) \kappa_2(-1) = -\chi^{(p)}(-1)$. Its prime-to-$p$ nebentype is $(\chi^{(p)})^{-1}$ (an irritating but inescapable consequence of our choices of conventions elsewhere).

    The construction depends only on the projection of $\Phi^{(p)}$ to the eigenspace where $\stbt{a}{0}{0}{a}$ for $a \in \widehat{\ZZ}^{(p)}$ acts as $\widehat{\chi}^{(p)}(a)^{-1}$, where $\widehat{\chi}^{(p)}$ is the adelic character attached to $\chi^{(p)}$ as above. We shall henceforth assume, without loss of generality, that $\Phi^{(p)}$ lies in this eigenspace; thus $\chi^{(p)}$ is uniquely determined by $\Phi^{(p)}$ and we sometimes drop it from the notation.
   \end{note}

   \begin{proposition}
    If $A$ is an affinoid algebra, and one of the $\kappa_i$ is a finite-order character, then $\cE^{\Phi^{p}}(\kappa_1, \kappa_2)$ is an overconvergent $A$-valued cusp form of weight-character $1 + \kappa_1 + \kappa_2$.
   \end{proposition}

   \begin{proof}
    Since twisting by a finite-order character preserves overconvergence, it suffices to assume $\kappa_1$ or $\kappa_2$ is 0. Then our $p$-adic Eisenstein series is the $p$-depletion of a family of \emph{ordinary} Eisenstein series, cf.~\cite[\S 2.3]{ohta99}, and it is well-known that these ordinary Eisenstein series are overconvergent (indeed, this is true by definition in Coleman's approach to overconvergent modular forms).
   \end{proof}

   As noted in \emph{op.cit.}, for $k \ge 1$, the Eisenstein series $F^{k}_{\Phi^p \Phi_{\mathrm{dep}}}$ described in \cite[\S 4.3]{LZ20b-regulator} is (the classical form associated to) the specialisation $\cE^{\Phi^p}(k-1, 0)$, and $E^{k}_{\Phi^p \Phi_{\mathrm{dep}}}$ is equal to $\cE^{\Phi^p}(0, k-1)$. It also implies the following relation:

   \begin{proposition}[cf.~{\cite[Prop 16.2.1]{LZ20b-regulator}}]
    Let $t \in \ZZ_{\ge 0}$. As overconvergent cusp forms of weight $-t$, we have
    \[ \theta^{-(1+t)}\left(F^{(t+2)}_{\Phi^p \Phi_{\mathrm{dep}}}\right) = \cE^{\Phi^p}(0, -1-t; \Phi^{(p)}), \]
    where $\theta = q \tfrac{\mathrm{d}}{\mathrm{d}q}$ is the Serre differential operator.
   \end{proposition}

  \subsection{Construction of the $p$-adic $L$-function} Let $A = \cO(\tilde{U}\times U')$, where $U' \subset \cW$ is an affinoid. Let $\bj:\ZZ_p^\times\rightarrow \cO(\cW)^\times$ be the universal character, and consider the following two $A$-valued characters of $T(\Zp)$: the canonical character $\nu_A = (\br_1, \br_2; \br_1 + \br_2)$, and the character $\tau_A = (\bt_1, \bt_2; \br_1 + \br_2)$, where
  \[ \bt_1 = \bj - 1, \qquad \bt_2 = \br_1 - \br_2 - \bj - 1.\]

  We fix a choice of test data $\gamma_S$ as in \cref{sect:testdata}. Since the spaces of higher Coleman theory of varying tame levels have an action of $G(\Af^p)$, we can make sense of $\gamma_{0, S} \cdot \uet$ as a family of classes at tame level $\widehat{K}^p$, which is still an eigenfamily for the Hecke operators away from $S$.

  \begin{definition}
   We define the family of Eisenstein series for $H$ defined by
   \[ \cG \coloneqq \cE^{\Phi_1^{(p)}}(0, \bj; \chi_1^{-1}) \boxtimes \cE^{\Phi_2^{(p)}}(\br_1-\br_2-\bj, 0; \chi_2^{-1}) \in H^0_{\id, \an}(\cS_{H, \Iw}(p^2), \tau_{A})^{+, \dag},\]
   where the tame level is taken to be $H \cap \widehat{K}^p$. (Note that this family is supported on the components of $U \times \cW$ of sign $\ep$, where $\ep = -\chi_1(-1)$.)
  \end{definition}

  \begin{definition}
   We let $\cL_{p, \gamma_S}(\upi; \uet, \chi_2)$ denote the element of $\cO(\tilde{U}\times U')$ defined by
   \[ \left\langle \hat\iota^*\left(\gamma_{0, S} \cdot \uet\right), \cG \right\rangle\]
   where the pairing
   \[ H^2_{\id, \an}(\cS_{H, \Iw}(p^t), \tau_A, \cusp)^{-, \dag} \times H^0_{\id, \an}(\cS_{H, \Iw}(p^t), \tau_A)^{+, \dag} \longrightarrow A\]
   is the one constructed in \cite[Theorem 9.2.1]{LZ21-erl}.
  \end{definition}

  These pairings are, by construction, compatible with base-change in $A$; so we may define $\cL_{p, \gamma_S}(\upi; \uet, \chi_2)$ as an element of $\cO(\tilde{U}\times \cW)$ by gluing over a family of affinoids whose union is $\cW$, even though $\cW$ itself is not affinoid. Note, however, that $\cL_{p, \gamma_S}(\upi; \uet, \chi_2)$ is not necessarily bounded in the $\cW$ variable.

\subsection{Interpolation properties}

 \begin{definition} Let $(P,Q) \in \tilde{U} \times \cW$.
   \begin{itemize}
    \item We say $(P, Q)$ is \emph{good critical} if $P$ is an algebraic weight $(r_1, r_2)$ with $r_1 \ge r_2 \ge -1$, and $Q = j + \rho$, where $j$ is an integer with $0 \le j \le r_1 -r_2$, and $\rho$ is a finite-order character.

    \item If instead we have $-1 - r_2 \le j \le -1$ and $\rho$ is trivial, we say $(P,Q)$ is \emph{good geometric}.
   \end{itemize}
  \end{definition}


 \subsubsection{Values in the critical range}

  Let $P \in \tilde{U}$ be good for $\upi$. Then the specialisation $\eta_P$ of $\uet$ defines Archimedean and $p$-adic Whittaker periods $\Omega_p^W(\pi_P, \eta_P)$ and $\Omega_\infty^W(\pi_P, \eta_P)$ (individually these are only defined up to $\overline{\QQ}^\times$, but their ratio is uniquely determined by $\eta$ as an element of $L \otimes_{\QQ(\pi_P)} \CC$).

  \begin{theorem}\label{thm:interpolation}
   The $p$-adic $L$-function $\cL_{p, \gamma_S}(\upi; \uet,\chi_2)$ has the following interpolation property. Suppose $P$ is good for $\upi$; and let $\ep = -\chi_1(-1) = (-1)^{r_1 - r_2+1}\chi_2(-1)$.

   If $r_1 = r_2$ and $\Lambda(\pi_P \times \chi_2, \tfrac{1}{2}) = 0$, then the restriction of $\cL_{p, \gamma_S}(\upi; \uet,\chi_2)$ to $\{P \} \times \cW^{\ep}$ is identically 0. In all other cases, the period $\Omega_{\pi_P}^{\ep}$ is defined and we have the following formula: for all $Q = j + \rho \in \cW^{\ep}$ such that $(P, Q)$ is good critical, we have
   \[
    \cL_{p, \gamma_S}(\upi; \uet,\chi_2)(P, Q) = C \cdot Z_S(\pi_P, \gamma_S)(Q)\cdot  R_p(\pi_P, \rho, j) \cdot \frac{\Lambda(\Pi_P \times \rho^{-1}, \tfrac{1-r_1+r_2}{2} + j)}{G(\rho^{-1})^2 \cdot  \Omega_{\pi_P}^{\ep}},
   \]
   where $C$ denotes the non-zero constant
   \[ C = \Omega_p^W(\pi_P, \eta_P) \cdot R_p(\pi_P, \chi_2, r_1 - r_2) \cdot \frac{\Lambda(\Pi_P \times \chi_2, \tfrac{r_1-r_2+1}{2}) \Omega_{\pi_P}^{\ep}}{G(\chi_2)^2 \Omega_\infty^W(\pi_P, \eta_P)} \in L^\times.\]
  \end{theorem}

  \begin{proof}
   Since the construction of $\cL_{p, \gamma_S}(\upi; \uet,\chi_2)$ is compatible with specialisation at $P$, the restriction to $\{P \} \times \cW^{\ep}$ is given by the cup-product
   \[ \left\langle \hat\iota^*\left(\gamma_{0, S} \cdot \eta_P \right), \cG \right\rangle,\]
   where $\eta_P$ is an eigenvector for the Iwahori-level Hecke operators $U'_{1, \Iw}$ and $U'_{2, \Iw}$.

   If we specialise at a $Q$ for which $\cG_Q$ is classical, then -- exactly as in \cite{LPSZ1} in the Klingen-ordinary case -- this cup-product can be written as a global period integral. This factorises into a product of local integrals at each place, which away from $S \cup \{p, \infty\}$ are are equal to the products of $L$-factors
   \[ L(\Pi_v \times \rho_v^{-1}, \tfrac{1-r_1+r_2}{2} + j) L(\Pi_v \times \chi_{2, v}, \tfrac{r_1 - r_2 + 1}{2}).\]
   The correction terms at the primes in $S$ give the factor $Z_S$. The factor at $p$ requires special treatment, since we are using slightly different local data here than in \cite{LPSZ1} (in particular, we are working with an Iwahori-level eigenvector, rather than Klingen-level); however, the computations of \cite{loeffler-zeta2} show that the integral for this new choice of local data still gives the Euler factor $R_p$. So we obtain a $p$-adic measure interpolating products of two $L$-values of $\Pi_P$, one varying with $Q$ and one fixed; and the ``fixed'' $L$-value gives the correction factor $C$.
  \end{proof}

  \begin{remark}
   The above theorem shows that the restriction of $\cL_{p, \gamma_S}$ to $\{P\} \times \cW^\ep$ has the same interpolating property as the 1-variable $p$-adic $L$-function $\cL_p^{\ep}(\pi_P)$ of Theorem \ref{thm:cycloL}, up to a constant depending on the choice of periods. However, we cannot immediately conclude that the two are the same, since the set of $Q$ at which the interpolation property applies is not Zariski-dense in $\cW^\ep$. This set \emph{is} dense in the algebraic spectrum $\Spec \Lambda_L(\Gamma)^{\ep}$, but we do not know \emph{a priori} that the restriction of the ``family'' $p$-adic $L$-function lies in the Iwasawa algebra -- it may only lie in the larger ring of unbounded rigid-analytic functions on $\cW^{\ep}$.

   It should be possible to show directly that the Iwahori-level eigenvector $\eta_P$, and the Klingen-level eigenvector used in \cite{LPSZ1}, both pair to the same value with any $p$-depleted overconvergent form on $H$ (whether or not this form is classical); this could be approached via identities relating Hecke operators on $G$ and on $H$, as in \S 12.5 of \cite{LZ20}. However, we shall establish this indirectly below as a by-product of our explicit reciprocity laws (see \cref{prop:compareL}).
  \end{remark}

  \begin{corollary}
   Suppose $P$ is good for $\upi$ with $r_2 \ne 0$. Let $\ep$ be a sign, and if $r_1 - r_2 = 0$, assume that $NV_{-\ep}(\pi_P)$ holds.

   Suppose that $j + \rho$ is a locally-algebraic character of sign $\ep$ such that $(P, j + \rho)$ is good critical and $L(\pi \times \rho^{-1}, \tfrac{1-r_1+r_2}{2} + j) \ne 0$. Then we may choose the characters $(\chi_1, \chi_2)$ and the test data $\gamma_S$ so that $\cL_{p, \gamma_S}(\upi; \uet,\chi_2)(P, j + \rho) \ne 0$.
  \end{corollary}

  \begin{proof}
   Using the hypothesis $NV_{-\ep}(\pi_P)$ in the parallel-weight case, we may choose the $\chi_i$ such that $-\chi_1(-1) = \ep$ and $\Lambda(\pi_P \otimes \chi_2, \tfrac{r_1 - r_2 + 1}{2}) \ne 0$. It suffices to note that the tame test data can be chosen so that $Z_S(\pi_P, \gamma_S) \ne 0$.
  \end{proof}

 \subsubsection{Values in the geometric range}

  Suppose $(P, j)\in \tilde{U} \times \cW$ is a point in the good geometric range; we let $(t_1, t_2) = (j-1, r_1-r_2-j-1)$ be the specialisations of $\bt_1, \bt_2$ at $(P, j)$, and we write
  \[ t_1' = -2-t_1 = -1-j.  \]
  Hence $0 \le t_1' \le r_2$, and the quadruple $(r_1, r_2, t_1', t_2)$ satisfies the branching law for algebraic representations defined in \cite[Proposition 6.4]{LPSZ1}, which is the condition needed to define motivic cohomology classes associated to $\pi_P$, using the pushforward of a pair of $\GL_2$ Eisenstein classes of weight $t_1'$ and $t_2$, respectively.

  \begin{remark}
   Note that we do not ``see'' all the values of the parameters this way: if we set $t_1' = r_1 - q-r$, $t_2 = r_2 - q + r$ following the notation of \cite{LZ20}, with $0 \le q \le r_2$ and $0 \le r \le r_1 - r_2$, then the constraint $t_1 + t_2 = r_1-r_2-2$ ends up forcing that $r = r_1 - r_2$; that is, the overconvergence condition implies we must take $r$ as large as possible. This is related to the fact that the ``extra'' $L$-value in \cref{thm:interpolation} is always at the upper end of the critical range (at $s = (r_1 - r_2 + 1)/2$).
  \end{remark}

  In \cite[\S 4]{LZ20b-regulator}, we defined an object $\operatorname{Per}_{\eta}(\pi_P, t)$ associated to $\pi_P $, the choice of twists $t_1'$ and $t_2$, and the basis vector $\eta_P \in S^2(\pi_P, L)$. This was a map $\mathcal{T}_S(\pi_P, L) \to L$ (satisfying an appropriate $H(\QQ_S)$-equivariance property), where
  \[ \mathcal{T}_S(\pi_P, L) = \cW(\pi_{P,S})_L \otimes \cS( \QQ_S^2, \widehat{\chi}_1, L) \otimes \cS( \QQ_S^2,  \widehat{\chi}_2, L).\]
  Evaluation at the new-vector $W^{\new}_{\pi_P}$ defines a map $L[G(\QQ_S)] \to \cW(\pi_{P, S})_L$, and hence
  \[ \cT_S(L) \to \cT_S(\pi_{P, S}, L). \]
  So we can regard $\operatorname{Per}_{\eta}(\pi_P)$ as a function on $\cT_S$, and we write $\operatorname{Per}_{\eta}(\pi_P, j,  \gamma_S) \in L$ for the value of $\operatorname{Per}_{\eta}(\pi_P )$ on some $\gamma_S\in \cT_S(L)$.

  \begin{proposition}
   \label{prop:coherentperiod}
   We have
   \[
    \cL_{p, \gamma_S}(\upi,\uet,\chi_2)(P, j) =  \operatorname{Per}_{\eta_P}(\pi_P, j, \gamma_S).
   \]
  \end{proposition}

  \begin{proof}
   By construction, we have
   \[
    \operatorname{Per}_{\eta_P}(\pi_P, j, \gamma_S) = \left\langle \iota^*_{\Kl}\left(\gamma_{0, S}\cdot \eta_{P, \Kl}\right), \theta^{-{(1 + t_1')}}\left(F^{(t_1' + 2)}_{\Phi_1^p \Phi_{\mathrm{dep}}}\right) \boxtimes F^{(t_2+2)}_{\Phi_2^p\Phi_{\mathrm{dep}}}\right\rangle, \]
   where $\iota_{\Kl}$ is an embedding of Shimura varieties at Klingen level. The term on the right-hand side is exactly the specialisation at $(P,t_2)$ of our family $\cG$ of $p$-adic modular forms for $H$. From the zeta-integral computations of \cite{loeffler-zeta2}, we may replace $\iota^*_{\Kl}\left(\eta_{\Kl}\right)$ with $\hat\iota^*\left(\eta_{\Iw}\right)$ without changing the value of the pairing.
  \end{proof}


\section{The explicit reciprocity law} \label{s:BSDpoints}

 \subsection{Setup} From now on, we make the following assumptions:

  \begin{assumptions}\label{ass:pidef}\
    \begin{enumerate}
     \item  Let $\pi$ be a \emph{deformable} cuspidal automorphic representation of $G$. We choose $\upi$, $\tilde{U}$ as in \cref{def:deformable} above.
     \item The family $\upi$ is not of (split) Yoshida type, so that for each classical specialisation $P$, the $\lambda_{P}$-eigenspace in \'etale cohomology of $S_{G, K^p\Iw(p)}$ is 4-dimensional.
     \item The family $\upi$ is ordinary at $p$.
    \end{enumerate}
\end{assumptions}

  By  \cite{tilouineurban99}, there is a  a family of Galois representations $V(\upi)$ associated to $\upi$, which is a rank 4 $\cO(\tilde{U})$-module with an action of $\Gal(\overline{\QQ}/\QQ)$, unramified outside $pN_0$ and satisfying $\operatorname{tr}(\operatorname{Frob}_\ell^{-1} | V(\upi)) = \lambda(T_{1,\ell})$ for $\ell \nmid pN_0$.

  \begin{notation}
      Write $\widehat{V(\upi)^*}$ for the reflexive hull of $V(\upi)^*$.
  \end{notation}

  Then $\widehat{V(\upi)^*}$ is locally free everywhere, and its specialisations in cohomological weights agree with that of $V(\upi)^*$.

 \subsection{Ordinary filtrations at $p$}

  The Galois representation $V(\upi)$ has a decreasing filtration by $\cO(U)$-submodules stable under $\Gal(\overline{\QQ}_p/\Qp)$ (via results of Urban \cite{urban05}; see \cite[Theorem 17.3.1]{LZ20} for the formulation we use). We write $\cF^i V(\upi)$ for the codimension $i$ subspace, and similarly for its dual $V(\upi)^*$. Note that $\Gr^0 V(\upi)^*$ is unramified, with arithmetic Frobenius acting as the $U_{\Sieg}$-eigenvalue. Abusing notation slightly\footnote{What we really mean is that $\Gr^1 V(\upi)^*$ is isomorphic to the tensor product of $\chi_{\mathrm{cyc}}^{(1 + \br_2)}$ and an unramified character.}, we may say that $\Gr^1 V(\upi)^*$ has ``Hodge--Tate weight $1 + \br_2$''.

  \begin{definition}
   We set
   \[\mathbb{V}^* = \widehat{V(\upi)}{}^* \mathop{\hat{\otimes}} \cO(\cW)(-1- \br_2 - \bj),\]
   which is a rank 4 locally free family of Galois representations over $\widetilde{U}\times \cW$. For a good weight $(P, Q)$ we write $\mathbb{V}_{P, Q}^*$ for the specialisation of $\mathbb{V}^*$ at $(P, Q)$.
  \end{definition}

  The quotient $\Gr^{1} \mathbb{V}^*$ has Hodge--Tate weight $-\bj$.
  As in \cite[\S 8.2]{KLZ17}, we can define a Coleman/Perrin-Riou big logarithm map for $\Gr^1\mathbb{V}^*$, which is a morphism of $\cO(\tilde{U} \times \cW)$-modules
  \[ \cL^{\mathrm{PR}}: H^1(\Qp, \Gr^{1} \mathbb{V}^*) \to \Dcris(\Gr^{1} \mathbb{V}^*). \]
  By construction, for good geometric weights $P$, this specialises to the Bloch--Kato logarithm map, up to an Euler factor; and for good critical weights it specialises to the Bloch--Kato dual exponential.

 \subsection{P-adic Eichler--Shimura isomorphisms}

  Let $P$ be a good weight. Then the Faltings--Tsuji comparison isomorphism of $p$-adic Hodge theory gives an identification between $\Dcris(V(\pi_P))$ and the $\pi_P$-eigenspace in de Rham cohomology (compatibly with the Hodge filtration); and the graded pieces of this filtration are identified with the coherent cohomology groups $S^i(\pi_P, L)$.

  Since the Hodge and Newton filtrations on $\Dcris$ must be complementary to each other (by weak admissibility), we deduce that there is an \emph{Eichler--Shimura} isomorphism
  \[
   \ES^2_{\pi_P}: S^2(\pi_P, L)\cong \Gr^{(r_2 + 1)}_{\mathrm{Hdg}} \Dcris(V(\pi_P))  \cong \Dcris(\Gr^2 V(\pi_P)).
  \]
  Concretely, the isomorphism is given by mapping an element in $\Gr^{(r_2 + 1)}_{\mathrm{Hdg}} \Dcris(V(\pi_P))$ to its unique lifting to $\Fil^{(r_2 + 1)}_{\mathrm{Hdg}}\Dcris(V(\pi_P)) \cap \ker( (\varphi - \alpha_P)(\varphi - \beta_P))$.

  \begin{remark}
   More generally, we have isomorphisms $\ES^i: S^i(\pi_P, L)\cong \Dcris(\Gr^i V(\pi_P))$ for each $0 \le i \le 3$, where $S^i(\pi_P, L)$ is the $\pi_P$-eigenspace in coherent $H^i$.

   We caution the reader that although the source and target of $\ES^i_{\pi_P}$ are the specialisations at $P$ of rank-one $\cO(\tilde{U})$-modules, it is \textbf{by no means obvious} that the isomorphisms $\ES^i_{\pi_P}$ for varying $P$ are the specialisations of a single $\cO(\tilde{U})$-module isomorphism ``$\ES^i_{\upi}$''. We shall establish (a slightly weakened form of) this below, under some additional hypotheses, as a by-product of our main Euler system argument.

   It would be very interesting to have a direct construction of the maps $\ES^i_{\upi}$ by methods of arithmetic geometry. For $i = 0$ (corresponding to classical holomorphic Siegel modular forms) this has been achieved in the recent preprint \cite{diao-rosso-wu21}. One can also obtain $\ES^3_{\upi}$ from this via Serre duality; but it seems to be more difficult to construct the ``intermediate'' filtration steps $i = 1, 2$.
  \end{remark}

 \subsection{Euler system classes}
  \label{sect:ESclass}

  Let $c, d$ be integers $> 1$ coprime to $6pS$. It follows from \cite[Theorem 7.1.1]{LRZ} that, associated to the data $\gamma_S$, we also have a family of cohomology classes
  \[ {}_{c, d}\bz_{m}(\upi, \gamma_S) \in H^1(\QQ(\mu_{m}), \mathbb{V}^*), \]
  for all square-free integers $m$ coprime to $p c d S$.

  \begin{remark}
   In \emph{op.cit}, the representations are parametrized in terms of the parameters $(a,b,q,r)$. These are related to our parametrisation by
   \[ a=r_1,\quad b=r_1-r_2,\quad q=r_1-t_2,\quad r=r_1-r_2.\qedhere\]
  \end{remark}

  The dependence on $c,d$ is as follows: set
  \[ E_{c, d} = (c^2 - c^{-\bt_1'}\chi_1(c))(d^2 - d^{-\bt_2}\chi_2(d)) \in \cO(\tilde{U} \times \cW). \]
  Then the element
  \[ \bz_1(\upi, \gamma_S) = {}_{c, d}\bz_1(\upi, \gamma_S) \otimes E_{c, d}^{-1} \in H^1(\QQ, \mathbb{V}^*) \otimes \operatorname{Frac} \cO(\tilde{U} \times \cW)\]
  is independent of $c, d$. More generally, one can find a lifting $E_{c, d,m}$ of $E_{c, d}$ to  $\cO(\tilde{U} \times \cW)[(\ZZ/m)^\times]$ for any $m$ coprime to $pcdS$, such that $\bz_m(\upi, \gamma_S) = {}_{c, d}\bz_1(\upi, \gamma_S) \otimes E_{c, d,m}^{-1}$ becomes independent of $m$.

  \begin{assumption}
   We assume henceforth that $\chi_2 \notin \{ 1, \chi_0^{-1}\}$, so that both of the $\chi_i$ are non-trivial.
  \end{assumption}

  Hence we can (and do) suppose that the $c, d$ are chosen such that $E_{c, d, m}$ is invertible, and hence define
  \[ \bz_m(\upi, \gamma_S) = E_{c, d}^{-1} \cdot {}_{c, d} \bz_m(\upi, \gamma_S) \in H^1(\QQ(\mu_m), \mathbb{V}^*).\]

  \begin{remark}
   In \cite{LZ20}, we took a slightly different approach to getting rid of the $(c, d)$-smoothing factor (since the methods used in \emph{op.~cit.}~constrained us to work at tame level 1, so the $\chi_i$ had to be trivial). Our approach here is rather closer to that of \cite{kato04}.
  \end{remark}

  \subsubsection*{Local properties} By construction, the image of $\bz_1(\upi, \gamma_S)$ under localisation at $p$ lands in the image of the (injective) map from the cohomology of $\cF^{1} \mathbb{V}^*$. So we may make sense of
  \[  \cL^{\mathrm{PR}}\left( \bz_1(\upi, \gamma_S) \right) \in \Dcris(\Gr^{1} \mathbb{V}^*), \]
  whose image at $(P, t_2)$ interpolates the Bloch--Kato logarithm (resp.~dual exponential) of $\bz_1(\upi, \gamma_S)(P, t_2)$ if $(P, t_2)$ is good geometric (resp.~good critical).

  Combining \cref{prop:coherentperiod} with \cite[Theorem 6.8.5]{LZ20} (c.f. also \cite[Theorem 4.5.1]{LZ20b-regulator}), which relates the periods $\operatorname{Per}_\eta(\dots)$ to the Euler system classes, we have the following result:

  \begin{theorem}\label{prop:geomreg}
   For each $(P,t_2)$ in the good geometric range, we have
   \[ \Big\langle \cL^{\mathrm{PR}}\left( \bz_1(\upi, \gamma_S) \right)(P,j), \mathrm{ES}^2_{\pi_P}(\eta_P)\Big\rangle = \cL_{p,\gamma_S}(\upi; \uet,\chi_2)(P, j). \]
  \end{theorem}

  More generally, the same argument also proves an ``equivariant'' version of this theorem for all $m$, relating $\bz_m(\upi, \gamma_S)$ to an equivariant $p$-adic $L$-function valued in the group ring of $(\ZZ/ m\ZZ)^\times$. This interpolates products of the form
  \[ \Lambda(\pi_P \otimes \rho, \tfrac{1-r_1+r_2}{2} + j) \Lambda(\pi_P \otimes \chi_2, \tfrac{r_1 - r_2 + 1}{2})
  \]
  for $\rho$ a character dividing $m$; note that we only twist the $L$-value depending on $j$ (the auxiliary term is \emph{not} twisted).

 \subsection{Reciprocity laws and meromorphic Eichler--Shimura}

  Let us now suppose that $\cL_{p, \gamma_S}(\upi, \uet)$ is not identically 0 (which we can always achieve by a suitable choice of $\gamma_S$).

  \begin{definition}
   Let $\fS(\upi)$ denote the set of points $P \in \tilde{U}$ which are good for $\upi$, with weight $(r_1, r_2) \in U \cap \ZZ^2$, and satisfy the following condition: there exists some $j \in \ZZ_{\ge 0}$, and some local data $\gamma_S$, such that $(P, j)$ is good geometric and $\cL_{p, \gamma_S}(\upi;\uet)$ is non-vanishing at $(P, j)$.
  \end{definition}

  \begin{lemma}
  The set $\fS(\upi)$ is Zariski-dense in $\tilde{U}$.
  \end{lemma}

  \begin{proof}
   Analogous to the proof of \cite[Lemma 11.6.2]{LZ21-erl}: any integer $j \le -1$ will be in the good geometric range for all sufficiently regular dominant weights $(r_1, r_2)$.
  \end{proof}

  Let us write $\cQ(U)$ for the fraction field of $\cO(U)$.

  \begin{theorem} \label{thm:explrecip}
    \
   \begin{enumerate}[(a)]
    \item There exists an isomorphism of $\cQ(U)$-modules
    \[ \ES^2_{\upi}: S^2(\upi) \otimes_{\cO(U)} \cQ(U) \cong \Dcris(\Gr^2 V(\upi)) \otimes_{\cO(U)} \cQ(U), \]
    depending only on $\upi$, characterised uniquely by the following property: for all $P = (r_1, r_2) \in \fS(\upi)$, the morphism $\ES^2_{\upi}$ is non-singular at $P$ and its fibre at $P$ coincides with the Eichler--Shimura morphism $\ES^2_{\pi_P}$.

    \item Extending $\ES^2_{\upi}$ to an isomorphism of $\cO(U\times\cW)$-modules, we have the explicit reciprocity law
    \[
     \left\langle
      \cL^{\mathrm{PR}}\left( \bz_m(\upi, \gamma_S)\right),
      \mathrm{ES}^2_{\upi}(\uet)\right\rangle = \cL_{p,\gamma_S}(\upi, \uet,\chi_2; \Delta_m).
    \]
   \end{enumerate}
  \end{theorem}

  \begin{proof}
   We start by choosing a ``random'' isomorphism $\jmath$ between $S^2(\upi)$ and $\Dcris(\Gr^2 V(\upi))$, which is possible since both are free rank 1 $\cO(U)$-modules.

   As in the proof of the preceding lemma, we choose local data $\gamma_S$ such that $\cL_{p, \gamma_S}(\upi; \uet,\chi_2)$ is not identically zero, and consider the ratio
   \[ \mathsf{R} = \frac{1}{\cL_{p, \gamma_S}(\upi, \uet,\chi_2)}  \left\langle\cL^{\mathrm{PR}}\left( \bz_{m}(\upi, \gamma_S)\right),
        \jmath(\uet)\right\rangle \in \cQ(U \times \cW).
   \]

   If we now take a $(P, Q)$ that is good geometric, and such that $\cL_{p, \gamma_S}(\upi; \uet,\chi_2)$ does not vanish at $(P, Q)$, it follows from the \cref{prop:geomreg} that $\mathsf{R}$ is regular at $(P, Q)$ and its value there is equal to the ratio $\jmath_P / \ES^2_{\pi_P}$ (independent of $Q$).

   We claim that $\mathsf{R} \in \cQ(U)$; that is, as a meromorphic function on $U \times \cW$, it is independent of the $\cW$ variable. To justify this, we argue as in Proposition 17.7.3 of \cite{LZ20}: we consider the meromorphic function $\mathsf{R}(\br_1, \br_2, \bt_2) - \mathsf{R}(\br_1, \br_2, \hat{\bt}_2)$ on $U \times \cW\times \cW$, where $\hat{\bt}_2$ is the coordinate on a second copy of $\cW$. Because of \cref{prop:geomreg}, this function has to vanish at all points $(r_1, r_2, t_2, \hat{t}_2)$ such that $(r_1,r_2, t_2)$ and $(r_1, r_2, \hat{t}_2)$ are both good geometric and neither is in the vanishing locus of $\cL_{p, \gamma_S}(\upi; \uet,\chi_2)$; this set is easily seen to be Zariski-dense in $U \times \cW\times\cW$. The same argument also shows that $\mathsf{R}$ doesn't depend on $\gamma_S$.

   Thus $\mathsf{R}$ is an element of $\cQ(U)^\times$, regular at all points $P \in \fS(\upi)$ and coinciding at each such point with the ratio $j_P / \ES^2_{\pi_P}$. So if we define $\ES^2_{\upi} = \mathsf{R}^{-1} \jmath$, then $\ES^2_{\upi}$ is regular at all points in $\fS(\upi)$ and coincides at such points with $\ES^2_{\pi_P}$. By the preceding lemma, this interpolating property uniquely determines $\ES^2_{\upi}$; and the reciprocity law holds by construction.
  \end{proof}

  \begin{remark}
   The set $\fS$ depends on the choice of characters $\uchi = (\chi_1, \chi_2)$. However, if we choose another pair of characters $\uchi'$, giving a possibily different set $\fS'$, then the same arguments show that the set $\fS \cap \fS'$ is Zariski-dense in $\tilde{U}$ and hence the map $\ES^2_{\upi}$ will be independent of this choice. Similarly, one can show that the above construction of $\ES^2_{\upi}$ agrees with the construction given in \cite[\S 11.6]{LZ21-erl} using $\GSp_4 \times \GL_2$ $p$-adic $L$-functions (under a slightly more restrictive assumption on the family implying $\tilde{U}= U$).
     \end{remark}

  \begin{corollary}
   \label{prop:compareL}
   For each good classical $P \in U$, the restriction of $\cL_{p,\gamma_S}(\upi, \uet,\chi_2; \Delta_m)$ to $\{P\} \times \cW^{\ep}$ is bounded; and if $P$ has weight $(r_1, r_2)$ with $r_2 \ne 0$, and Hypothesis $NV_{-\ep}(\pi_P)$ holds, then this restriction is given by $C Z_S(\pi_P, \gamma_S) \cdot \cL_p^{\ep}(\pi_P)$, where $C$ is as in \cref{thm:interpolation} and $\cL_p^{\ep}(\pi_P)$ is as defined in \cref{sect:cycloL}.
  \end{corollary}

  \begin{proof}
   Since the Perrin-Riou big logarithm (for a one-dimensional $\Qp$-valued representation) takes values in the Iwasawa algebra, it follows from \cref{thm:explrecip} that the restriction of $\cL_{p,\gamma_S}(\upi; \uet,\chi_2)$ to $\{P\} \times \cW^{\ep}$ is bounded. Since we know it agrees with  $C Z_S(\pi_P, \gamma_S) \cdot \cL_p^{\ep}(\pi_P)$ at all finite-order characters, these two measures are identical.
  \end{proof}

\section{The leading-term argument}
 \label{sect:leadingterm}
 \subsection{The leading-term argument}

  Let us now choose a $\pi$ which is \emph{deformable} in the sense of Definition \cref{def:deformable}, of some weight $(r_1, r_2)$; note that we do not require $(r_1, r_2)$ to be cohomological. For simplicity, we suppose $r_1 - r_2$ is even (which implies $\chi_0(-1) = 1$).

  We shall need to suppose that for some (and hence every) $G_{\QQ}$-stable $\cO$-lattice $T$ in $V_p(\pi)^*$, the lattice $T$ satisifies the usual ``big image'' conditions (H.0)--(H.4) of \cite{mazurrubin04}, and so does $T(\eta)$ for every Dirichlet character $\eta$ of $p$-power order and square-free conductor coprime to $pS$.

  We shall also choose a sign $\ep$ such that the hypotheses of \cref{thm:cycloL} are satisfied, and such that the $p$-adic $L$-function $\cL_p^{\ep}(\pi)$ is not identically 0. It follows that we may find the following auxiliary data:
  \begin{itemize}
   \item A pair of Dirichlet characters $(\chi_1, \chi_2)$ of prime-to-$p$ conductors, both non-trivial and with $\chi_0 \chi_1 \chi_2 = 1$, and such that $\chi_1(-1) = -\ep$ and $L(\Pi \times \chi_2, \tfrac{r_1-r_2+1}{2}) \ne 0$.
   \item Two integers $c,d > 1$ coprime to $6pS$ and such that $E_{c, d}$ is invertible in $\Lambda_L(\Gamma)^\ep$.
   \item Local test data $\gamma_S$ with $Z_S(\pi, \chi_2, \gamma_S)$ invertible in $\Lambda_L(\Gamma)^{\ep}$.
  \end{itemize}

  \begin{remark}
   It is quite irritating that we have to assume non-vanishing of the auxiliary twist by $\chi_2$. If this $L$-value does vanish, then we obtain Selmer group bounds in terms of the ``leading term'' of $\cL_{p, \gamma_S}(\upi, \uet,\chi_2)(P', -)$ as $P' \to P$; but it seems difficult to relate this ``leading-term $p$-adic $L$-function'' to the actual $L$-values of $\pi$, unless we have some independent \emph{a priori} construction of a $p$-adic $L$-function interpolating $L$-values of $\pi$ without the auxiliary twist. Such a construction is available if $\pi$ is cohomological, via lifting to $\GL_4$ (although this is only written up under a restrictive assumption on the levels). Alternatively, if $\pi$ is a $\theta$-lift from $\GL_2 / K$ with $K$ real quadratic, we can use modular symbols over $K$, as in \cite{LZ20-yoshida}.
  \end{remark}

  We now apply the axiomatic leading-term argument developed in \cite{LZ20-yoshida}, leading up to the proof of Theorem 10.3.3 of \emph{op.cit.}. This shows that, for any choice of the period $\Omega_{\pi}^{\ep}$, there exists a family of classes
  \[ c_m \in H^1_{\Iw}(\QQ_\infty[m], V_p(\pi)^*(-1-r_2))^{\ep}\]
  (landing in the $\ep$-eigenspace for complex conjugation) with the following property: the $c_m$ all land in the cohomology of some lattice independent of $m$, and the image of $c_m$ under the Perrin-Riou regulator is equal to $\cL_p^{\ep}(\pi; \Delta_m)$.

  \begin{theorem}
   \label{thm:IMCautorep}
   Under the above hypotheses, if $\cL_p^{\ep}(\pi)$ is defined using an optimally normalised period, we have the divisibility of ideals in $\Lambda_{\cO}(\Gamma)^{\ep}$ (``half'' of the Iwasawa main conjecture in the $\ep$-eigenspace):
   \[ \operatorname{char}_{\Lambda_{\cO}(\Gamma)^{\ep}} \widetilde{H}^2(\QQ_\infty, T_p(\pi)^*(-1-r_2))\  \big|\ \cL_p^{\ep}(\pi).\]
  \end{theorem}

  \begin{proof}
   This follows by exactly the same method as Theorem 12.2.2 of \cite{LZ20-yoshida}; the assumption that our period be optimally normalised means that the comparison term $\mu_{\mathrm{min}, u}$ appearing in \emph{op.cit.} is zero.
  \end{proof}

  As a corollary, we obtain Theorem \ref{thm:mainBK} of the introduction:

  \begin{theorem}
   \label{thm:H1f}
   If $j$ is an integer with $0 \le j \le r_1 - r_2$ and $(-1)^j = \ep$, and $L(\Pi, \tfrac{1-r_1+r_2}{2} + j) \ne 0$, then $H^1_{\mathrm{f}}(\QQ, V_p(\pi)^*(-1-r_2 - j)) = 0$.
  \end{theorem}

  \begin{proof}
   This follows from the bound we have established for Selmer groups over $\QQ_\infty$, together with descent results comparing Selmer complexes over $\QQ_\infty$ and over $\QQ$.
  \end{proof}

\subsection{Application to the Birch--Swinnerton-Dyer conjecture}

   Let $A/\QQ$ be a modular abelian surface without extra endomorphisms, and let $\pi$ be the globally generic cuspidal automorphic representation of $\GSp_4$ of weight $(r_1,r_2)=(-1,-1)$ corresponding to $A$. Assume that $\pi$ satisfies the following conditions:
  \begin{enumerate}
    \item the representation $\pi$ is \emph{deformable} in the sense of \cref{def:deformable};

   \item $\pi$ is ordinary at $p$;

   \item The Galois representation $V = V_p(\pi)^*$ ($\cong V_pA$) satisfies the ``big image'' conditions of \cite[\S 3.5]{mazurrubin04};

   \item there exists an odd Dirichlet character $\chi$ such that $L(A, \chi, 1) \ne 0$.
  \end{enumerate}

  \begin{remark}
   Note that the ``big image'' condition is satisfied if $A$ is of general type and the Galois image is $\GSp_4(\Zp)$. It is also satisfied if $A$ is the restriction of scalars of an elliptic curve $E$ over an imaginary quadratic field $K$ and the image of $\Gal(\overline{K} / K)$ on $T_p(E) \times T_p(E^{\sigma})$ contains $\SL_2(\Zp) \times \SL_2(\Zp)$, where $E^{\sigma}$ is the conjugate of $E$ by $\Gal(K / \QQ)$; this is true for all but finitely many $p$ if $E$ is non-CM and not isogenous to $E^\sigma$.
  \end{remark}

  \begin{theorem}
  \label{thm:BSD}
   In the above setting, if $L(A,1) \neq 0$, then
   \[ \operatorname{rank}_{\ZZ}\, A(\QQ)= 0,\]
   as predicted by the Birch--Swinnerton-Dyer conjecture; furthermore, $\Sha_{p^\infty}(A / \QQ)$ is finite.
  \end{theorem}

  \begin{proof}
   This is exactly \cref{thm:H1f} applied to the automorphic representation $\pi$ associated to $A$.
  \end{proof}

  \begin{remark}
   In order to obtain one divisibility in the exact BSD leading-term formula for $A$, we would need to know how to compare the ``optimally normalised'' period defined here with the period of a Ner\'on differential form on $A$. It does not appear to be obvious how to do this.
  \end{remark}

  We also obtain one inclusion of the Iwasawa main conjecture: let $\widetilde{R\Gamma}_{\Iw}(\QQ_\infty, V_p(A))$ denote the Nekovar Selmer complex, with the unramified local conditions at $\ell\neq p$, and at $p$ the Greenberg-type local condition determined by $\Fil^2 V_\pi^*$.

  \begin{theorem}\label{thm:IMC}
      Assume that Conditions (1)--(4) above are satisfied, and that $L(A,1)\neq 0$. Then $\widetilde{H}^1_{\Iw}(\QQ_\infty,V_pA)=0$, and  $\widetilde{H}^2_{\Iw}(\QQ_\infty,V_pA)$ is a finitely-generated torsion module over $\Lambda(\ZZ_p^\times)$, whose characteristic ideal divides $\cL_{p}(\pi)^{\ep}$.
  \end{theorem}
  \begin{proof}
      An application of the Euler system machine (c.f. \cite[Theorem 18.1.3]{LZ20}).
  \end{proof}


\appendix

\section{Remarks on the deformability hypothesis}

 The aim of this section is to discuss evidence for the ``deformability'' condition (\ref{def}) in our main theorem, and relate it to existing conjectures on the geometry of eigenvarieties.

 \subsection{The generic case}

  Let $A$ be a semistable, modular abelian surface with $\End_{\QQbar}(A) = \ZZ$.

  \subsubsection{Selmer vanishing} We recall the following result of Calegari--Geraghty--Harris (from the Appendix to \cite{calegarigeraghty20}):

  \begin{theorem}
  There exists a density one set of primes $p$ such that the following hold:
   \begin{itemize}
    \item $A$ is good ordinary at $p$;
    \item the Galois representation $V_p(A)$ is residually irreducible;
    \item the unit eigenvalues $\alpha, \beta$ of Frobenius on the Dieudonn\'e module of $A$ are distinct mod $p$;
    \item we have $H^1_{\mathrm{f}}(\QQ, \ad^0 (V_p A)) = 0$, where $\ad^0 (V_p A)$ denotes the 10-dimensional $\operatorname{PGSp}_4$ adjoint representation.
   \end{itemize}
  \end{theorem}

  We let $p$ be a prime for which the conclusions of the theorem hold. It is important to note that the adjoint motive of $A$ is \emph{not} critical, so we do not also obtain vanishing of the Tate dual $\ad^0(V_pA)(1)$ (in contrast to the case of elliptic curves over totally-real fields); instead, $H^1_{\mathrm{f}}(\QQ, \ad^0\left(V_p A\right)(1))$ is one-dimensional.

  \subsubsection{Deformation spaces} Let $\bar\rho$ be the dual of the residual representation of $V_p(A)$, and $R$ a complete Noetherian local $\Zp$-algebra. We consider deformations of $\rb$ as a $\Gamma_{\QQ, S}$-representation, where $S$ is the set of primes where $\rho$ is ramified. All our deformations will be symplectic, with multiplier equal to the inverse cyclotomic character. Let us write $\operatorname{ur}(x)$ for the unramified character mapping geometric Frobenius to $x$.

  \begin{definition}
   A deformation $\rho$ of $\rb$ to $R$ is \emph{Klingen-ordinary} if at $p$ it has the form
   \[
    \begin{pmatrix}
     \operatorname{ur}(u) & 0 & \star & \star\\
     0 & \operatorname{ur}(v) & \star & \star \\
     0 & 0 & \ddots \\
     0 & 0 &
    \end{pmatrix}
   \]
   for some $u, v \in R^\times$ lifting $\bar{\alpha}, \bar{\beta}$ respectively. It is \emph{Borel-ordinary} if it has the shape
   \[
    \begin{pmatrix}
     \operatorname{ur}(u) & \star & \star & \star\\
     0 & \operatorname{ur}(v) & \star & \star \\
     0 & 0 & \ddots \\
     0 & 0 &
    \end{pmatrix}
   \]
   for some $u, v$ (note the extra star). It is \emph{weakly Borel-ordinary} if it has the form
   \[
    \begin{pmatrix}
     \tau_1 & \star & \star & \star\\
     0 & \tau_2 & \star & \star \\
     0 & 0 & \ddots \\
     0 & 0 &
    \end{pmatrix}
   \]
   for characters $\tau_1, \tau_2$ lifting $\operatorname{ur}(\bar\alpha)$, $\operatorname{ur}(\bar\beta)$ respectively (but we do not assume the $\tau_i$ to be unramified). Finally, it is \emph{weakly Klingen-ordinary} if it has the form
   \[
    \begin{pmatrix}
     \tau_1 & 0 & \star & \star\\
     0 & \tau_2 & \star & \star \\
     0 & 0 & \ddots \\
     0 & 0 &
    \end{pmatrix}
   \]
   with $\tau_1, \tau_2$ lifting $\operatorname{ur}(\bar\alpha)$, $\operatorname{ur}(\bar\beta)$, as before, and the \emph{ratio} of the $\tau_i$ is unramified.
  \end{definition}

  For each of these deformation problems, we can define a global deformation space parametrising deformations of $\bar\rho$ which satisfy the given local condition at $p$, and the ``minimal'' local condition at primes away from $p$. We denote the rigid-analytic generic fibres of the corresponding deformation spaces by $\cX_{\rb}^{?}$, with $? = \Kl, \B, \wKl, \wB$. The logical implications among our deformation conditions give maps
  \[\begin{tikzcd}[row sep=small]
   & \cX_{\rb}^\B \drar[]\\
   \cX_{\rb}^{\Kl} \drar[]
   \urar[] & & \cX_{\rb}^\wB.\\
   & \cX_{\rb}^{\wKl}\urar[]
  \end{tikzcd}
  \]

  By construction, $\cX_{\rb}^{\wB}$ maps naturally to the product of two copies of the unit disc (parametrising characters of the inertia subgroup of $\Gamma_{\Qp}^{\mathrm{ab}}$) and the fibre over the identity is $\cX_{\rb}^{\B}$. Similarly, $\cX_{\rb}^{\wKl}$ maps to one copy of the unit disc and the fibre over the identity is $\cX_{\rb}^{\Kl}$.

  \begin{remark}
   Note that $\cX_{\rb}^{\wKl}$ is contained in the preimage in $\cX_{\rb}^{\wB}$ of the ``parallel weight'' locus (where $\tau_1 / \tau_2$ is unramified). However, the containment may be strict, because of the extra $\star$ in the $(1, 2)$ position for $? = \wB$.
  \end{remark}

  \subsubsection{Tangent spaces} The deformation $\rho = (V_p A)^*$ (and a choice of labelling of the unit-root Frobenius eigenvalues) determines a point of $\cX_{\rb}^{\Kl}$, and hence of the other spaces as well. By standard methods, we can compute the tangent space at this point in terms of Galois cohomology:

  \begin{proposition}
   The tangent space of $\cX_{\rb}^{?}$ at $\rho$ is given by a Selmer group
   \[ H^1_?(\QQ, \ad^0(\rho)) \coloneqq \operatorname{ker}\left(H^1(\GQS, \ad^0(\rho)) \longrightarrow \bigoplus_{\ell \in S} H^1_?(\QQ_\ell, \ad^0(\rho))\right),\]
   where $H^1_?(\QQ_\ell, \ad^0(\rho)_L) \subset H^1(\QQ_\ell, \ad^0(\rho))$ is a subspace determined by the local condition $?$ at $\ell$. For $\ell \ne p$, we have $H^1_?(\QQ_\ell, \ad^0(\rho))  = H^1_{\mathrm{f}}(\Ql,  \ad^0(\rho))$ (whatever the value of $?$).
  \end{proposition}

  So our tangent-space Selmer groups differ only in the local condition at $p$. The following is shown in \S 4 of \cite{calegarigeraghty20}:

  \begin{proposition}
   We have $H^1_{\Kl}(\Qp, \ad^0 \rho) = H^1_{\mathrm{f}}(\Qp, \ad^0 \rho)$, so (by our assumptions on $A$ and $p$) the tangent space of $\cX_{\rb}^{\Kl}$ at $\rho$ is zero.
  \end{proposition}

  Let us write $H^1_?(\Qp, \ad^0(\rho)(1))$ for the orthogonal complement of $H^1_?(\Qp, \ad^0 \rho)$ under Tate duality. If $? \in \{\wKl, \B, \wB\}$, then there are localisation maps
  \[ (\dag)\qquad H^1_{?}(\QQ, \ad^0 \rho) \to \frac{H^1_{?}(\Qp, \ad^0 \rho)}{H^1_{\mathrm{f}}(\Qp, \ad^0 \rho)} \qquad\text{and}\qquad  \frac{H^1_{\mathrm{f}}(\Qp, \ad^0 \rho(1))}{H^1_{?}(\Qp, \ad^0 \rho(1))} \leftarrow H^1_{\mathrm{f}}(\QQ, \ad^0 \rho(1)) \qquad (\ddag) \]
  and the Poitou--Tate global duality theorem shows that the images of these maps are orthogonal complements (with respect to the perfect pairing between the middle two spaces induced by local Tate duality).

  \begin{definition}
   We say the deformation problem $?$ is \emph{unobstructed} if the right-hand map $(\ddag)$ is nonzero (and hence injective), so $H^1_?(\QQ, \ad^0 \rho(1)) = 0$.
  \end{definition}

  We can compute the dimensions of the local-condition spaces in much the same way as Lemma 4.8 of \cite{calegarigeraghty20}, and we obtain the following formula:
  \[
   \dim H^1_?(\QQ, \ad^0(\rho)_L) = \delta_? +
   \begin{cases}
    -1 & \text{if $? = \Kl$},\\
     0 & \text{if $? = \B$ or $\wKl$},\\
     2 & \text{if $? = \wB$}.
   \end{cases}
  \]
  where $\delta_?$ is 1 if the deformation problem is obstructed and 0 if it is unobstructed. (Clearly, $? = \Kl$ is always obstructed.) Of course, if either of the deformation problems $? = \B$ or $? = \wKl$ is unobstructed, then $? = \wB$ is \emph{a fortiori} unobstructed.

  \begin{remark}
   There seems to be ``no particular reason'' for the image of $H^1_{\mathrm{f}}(\QQ, \ad^0 \rho(1))$ in the local cohomology $H^1_{\mathrm{f}}(\Qp, \ad^0 \rho(1))$ to land in the much smaller subspace $H^1_{\mathrm{wB}}(\Qp, \ad^0 \rho(1))$. So the conjecture that the $\wB$ Selmer group should always be unobstructed seems very plausible.
  \end{remark}

 \subsection{Application to deformability}

  \begin{proposition}
   If the deformation problem $? = \wB$ is unobstructed, then $\rho$ is deformable in the sense of \cref{def:deformable}.
  \end{proposition}

  \begin{proof}
   Let $\cE$ denote the $\GSp_4$ eigenvariety of the appropriate tame level, as above. Then the automorphic representation $\pi$ associated to $A$ defines a point of $\cE$. Moreover, we can find a deformation of the Galois representation $\rho$ to a neighbourhood of $\pi$: that is, there exists a neighbourhood $\cN$ of $\rho$, and a homomorphism $\cN\to \cX^{\wB}_{\rb}$, mapping classical points to their associated Galois representations (and, in particular, sending $\pi$ to $\rho$). By construction, this map intertwines the weight map $\cN\subset \cE\to \cW^2$ with the natural map from $\cX^{\wB}_{\rb}$ to two copies of the unit disc.

   The map $\cN\to \cX_{\rb}^{\wB}$ must be injective, since the Hecke eigensystems associated to automorphic representations are determined by their Galois representations. So, if the deformation problem $? = \wB$ is unobstructed, then the tangent space of $\cE$ at $\rb$ has dimension $\le 2$ (the dimension of the tangent space of $\cX_{\rb}^{\wB}$); and the relative tangent space over weight space has dimension $\le 1$ (the dimension of the tangent space of $\cX_{\rb}^{\B}$, which is the fibre of $\cX_{\rb}^{\wB}$ over weight space).

   Since $\rho$ is irreducible, $\cE$ must be flat over weight space around $\rho$ and thus its tangent space cannot have dimension $< 2$. So it is smooth at $\rho$ and the differential of the weight map is non-zero, as required.
  \end{proof}

  \begin{remark}
   We expect, of course, that the map from $\cE$ to the deformation space should be an isomorphism (an ``R = T'' conjecture for Borel-ordinary families). If this holds (at least locally around $\rho$), then the converse of the above proposition is also true -- i.e.~deformability is \emph{equivalent} to the $\wB$ deformation problem being unobstructed.
  \end{remark}

 \subsection{Relation to the Klingen partial eigenvariety}

  Let $\mathbf{T}^{\Kl}_{\rb}$ denote the Hecke algebra acting on the $\rb$-localisation of the Klingen-ordinary cohomology complex of \cite{pilloni20}; and let $\cE^{\Kl}_{\rb}$ denote the generic fibre of its formal spectrum. As in Theorem 7.9.4 of \cite{BCGP}, one can construct a deformation of $\rb$ valued in $\mathbf{T}^{\Kl}_{\rb}$, compatible with the Galois representations associated to classical specialisations of $\mathbf{T}^{\Kl}_{\rb}$; and this deformation satisfies the ``weak Klingen-ordinary'' deformation condition, so we obtain a map $\cE_{\rb}^{\Kl} \to \cX_{\rb}^{\wKl}$.

  \begin{conjecture}[``R = T'' for Klingen-ordinary deformations]\label{rmk:RT}
   This deformation \textbf{is universal}, i.e.~it induces an isomorphism between the universal weakly-Klingen-ordinary deformation ring and the Hecke algebra $\mathbf{T}^{\wKl}_{\rb}$.
  \end{conjecture}

  This conjecture ought to be accessible using the methods of \cite{BCGP} (although it is a little stronger than the $R = T$ results actually proved in \emph{op.cit.}). If this holds, then we can identify $\cX_{\rb}^{\wKl}$ with the tangent space of $\cE_{\rb}^{\wKl}$ at $\rho$.

  \begin{corollary}
   If the above conjecture holds, then the deformation problem $? = \wKl$ is unobstructed if and only if $\rho$ is an isolated and reduced point of $\cE_{\rb}^{\wKl}$.
  \end{corollary}

  In particular, if $\mathbf{T}^{\wKl}_{\rb}[1/p]$ is finite and \'etale over $\Qp$, then $? = \wKl$ is unobstructed. It is conjectured in the introduction of \cite{boxerpilloni20} that $\mathbf{T}^{\wKl}_{\rb}$ should be torsion over the Iwasawa algebra, and hence its generic fibre finite over $\Qp$, whenever $\rb$ is not induced from a 2-dimensional representation over a real quadratic field. So it seems natural to expect that $\mathbf{T}^{\wKl}_{\rb}[1/p]$ should also be \'etale for ``most'' primes $p$; and, as we have seen, this implies deformability in the sense of the main text.

 \subsection{Lifts from quadratic fields}

  We now consider the case when $A = \Res_{K / \QQ}(E)$ where $K$ is an imaginary quadratic field. This case is not immediately covered by the results of \cite{calegarigeraghty20}, but it should be within reach via the same methods to show that $H^1_{\mathrm{f}}(\QQ, \ad^0 (V_p A)) = 0$ for many primes $p$; in any case this is predicted by the Bloch--Kato conjecture for the adjoint. We shall assume this for the remainder of this section.

  \begin{proposition}
   Let $\tau$ be the 2-dimensional $G_K$-representation from which $\rho$ is induced. Then we have
   \[ \ad^0(\rho) = \Ind_K^{\QQ}(\ad^0 \tau) \oplus \As(\tau)(1), \]
   where $\As(\tau)$ denotes the 4-dimensional Asai (tensor-induction) representation.
  \end{proposition}

  This gives decompositions of the source and target of the map $(\ddag)$, for each deformation problem $?$, into a ``$\GL_2/K$ adjoint'' part and an ``Asai'' part; and the map $(\ddag)$ itself preserves this decomposition.

  \begin{corollary}
   If $K$ is imaginary quadratic, then the $? = \B$ deformation problem is \emph{always} obstructed.
  \end{corollary}

  \begin{proof}
   By Poitou--Tate duality, the Selmer group of $\Ind_{K/\QQ} \ad^0(\tau)(2)$ has to have dimension $\ge 1$. So the source of the map $(\ddag)$ lies entirely in the ``adjoint'' summand. However, for $? = \B$ the target of the map lies entirely in the ``Asai'' summand. Hence the map $(\ddag)$ is the zero map.
  \end{proof}

  So, if the Borel-ordinary $R = T$ conjecture outlined in Remark \ref{rmk:RT} holds, then the eigenvariety is \emph{always} ramified over weight space at points lifted from elliptic curves over imaginary quadratic fields.

  However, it is $? = \wB$ which is important for our applications; and the difference between the $? = \wB$ and $? = \B$ local conditions lies entirely in the ``$\GL_2 / K$ adjoint'' summand. Hence it is related to nearly-ordinary infinitesimal deformations of $\tau$. The assumption that the Bloch--Kato Selmer group vanishes means that these deformations are classified by their \emph{infinitesimal Hodge--Tate weights} in the sense of \cite{calegarimazur09}.

  \begin{propqed}
   In the imaginary-quadratic case, the $? = \wB$ deformation problem is unobstructed if and only if the infinitesimal Hodge--Tate weights of all nearly-ordinary infinitesimal deformations of $\tau$ span a proper subspace of $K \otimes \Qp$.
  \end{propqed}

  Assuming $p$ splits in $K$, Calegari and Mazur have conjectured \cite[Conjecture 1.3]{calegarimazur09} that $\tau$ cannot admit any infinitesimal deformation whose infinitesimal Hodge--Tate weight lies in $K \subset K \otimes \Qp$. So the Calegari--Mazur conjecture certainly implies that $? = \wB$ is always unobstructed for split primes $p$, and hence that deformability holds for all split $p$.

  \begin{remark}
   It seems natural to conjecture that the $\GL_2 / K$ eigenvariety should lift to a subspace of the $\GSp_4 / K$ eigenvariety, and the $\GSp_4 / \QQ$ eigenvariety should ramify over weight space along the image of this map. However, even verifying the existence of the lifting is difficult, since classical points are not dense in the $\GL_2 / K$ eigenvariety and hence we cannot deduce this from the analogous lifting result for classical automorphic forms.
  \end{remark}

  \begin{remark}
   For lifts from real quadratic fields, the numerology is different: the ``obstruction class'' in the cohomology of $\ad^0(\rho)(1)$ is forced to live in the Asai summand, not the adjoint summand. This means that the $\wKl$ deformation problem is automatically obstructed, which accounts for the observation in \cite{pilloni20} that real-quadratic lifts give components of the Klingen partial eigenvariety with larger-than-expected dimension.

   Since the deformation theory of the $\GL_2 / K$ adjoint is unobstructed, we obtain an entire component of the $\GSp_4$ eigenvariety passing through $\pi$ consisting of forms lifted from $\GL_2 / K$. In this case, $? = \wB$ is obstructed if and only if $? = \B$ is, and (assuming the expected $R = T$ conjecture) this holds precisely when there exist additional, non-lift components pasing through $\pi$.
  \end{remark}



\providecommand{\bysame}{\leavevmode\hbox to3em{\hrulefill}\thinspace}
\providecommand{\MR}[1]{}
\renewcommand{\MR}[1]{%
 MR \href{http://www.ams.org/mathscinet-getitem?mr=#1}{#1}.
}
\providecommand{\href}[2]{#2}
\newcommand{\articlehref}[2]{\href{#1}{#2}}

\end{document}